\documentclass{amsart}
\usepackage{amscd,amsmath,amsthm,amssymb}
\usepackage[utf8]{inputenc}
\usepackage[dvipsnames]{xcolor}
\usepackage{enumerate}
\usepackage{graphicx}
\usepackage{tikz}
\usetikzlibrary{shapes.geometric}
\usetikzlibrary{fit}
\usepackage{cleveref}
\usepackage{float}
\usepackage{mathrsfs}
\usepackage{amsaddr}

 \def\frk{\mathfrak}               

 \def\mm{{\frk m}}

 \def\G{{\mathcal G}}

 \def\xb{{\mathbf x}}

 \def\fb{{\mathbf f}}
 \def\sb{{\mathbf s}}

 \def\opn#1#2{\def#1{\operatorname{#2}}} 
 \opn\chara{char} \opn\length{\ell} \opn\pd{pd} \opn\rk{rk}
 \opn\projdim{proj\,dim} \opn\injdim{inj\,dim} \opn\rank{rank}
 \opn\depth{depth} \opn\grade{grade} \opn\height{height}
 \opn\embdim{emb\,dim} \opn\codim{codim}
 
 \opn\Tr{Tr} \opn\bigrank{big\,rank}
 \opn\superheight{superheight}\opn\lcm{lcm}
 \opn\trdeg{tr\,deg}
 \opn\reg{reg} \opn\lreg{lreg} \opn\ini{in} \opn\lpd{lpd}
 \opn\size{size} \opn\sdepth{sdepth}
 \opn\HS{HS}
  
  \opn\link{link}\opn\fdepth{fdepth}\opn\lex{lex}
 \opn\tr{tr}
 \opn\type{type}
 \opn\gap{gap}
 \opn\arithdeg{arith-deg}
 \opn\revlex{revlex}
 \opn\div{div} \opn\Div{Div} \opn\cl{cl} \opn\Cl{Cl}
 \opn\Spec{Spec} \opn\Supp{Supp} \opn\supp{supp} \opn\Sing{Sing}
 \opn\Ass{Ass} \opn\Min{Min}\opn\Mon{Mon}
 \opn\Ann{Ann} \opn\Rad{Rad} \opn\Soc{Soc}
 \opn\Im{Im} \opn\Ker{Ker} \opn\Coker{Coker} \opn\Am{Am}
 \opn\Hom{Hom} \opn\Tor{Tor} \opn\Ext{Ext} \opn\End{End}
 \opn\Aut{Aut} \opn\id{id}
 
 \opn\nat{nat}
 \opn\pff{pf}
 \opn\Pf{Pf} \opn\GL{GL} \opn\SL{SL} \opn\mod{mod} \opn\ord{ord}
 \opn\Gin{Gin} \opn\Hilb{Hilb}\opn\sort{sort}
 \opn\PF{PF}\opn\Ap{Ap}
 \opn\mult{mult}
 \opn\bight{bight}
 \opn\aff{aff}
 \opn\relint{relint} \opn\st{st}
 \opn\lk{lk} \opn\cn{cn} \opn\core{core} \opn\vol{vol}  \opn\inp{inp} \opn\nilpot{nilpot}
 \opn\link{link} \opn\star{star}\opn\lex{lex}\opn\set{set}
 \opn\width{wd}
 \opn\Fr{F}
 \opn\QF{QF}
 \opn\G{G}
 \opn\type{type}\opn\res{res}
 \opn\conv{conv}
 \opn\Deg{Deg}
 \opn\Sym{Sym}
 \opn\gr{gr}
 
 \def\pot#1#2{#1[\kern-0.28ex[#2]\kern-0.28ex]}

 \opn\dirlim{\underrightarrow{\lim}}
 \opn\inivlim{\underleftarrow{\lim}}
 \let\iso=\cong

 \def\iso{\cong}
 \def\rank{\operatorname{rank}}
 \def\reg{\operatorname{reg}}
 \def\sort{\operatorname{sort}}
 \def\ini{\operatorname{in}}

 \newtheorem{Theorem}{Theorem}[section]
 
 \newtheorem{Corollary}[Theorem]{Corollary}
 \newtheorem{Proposition}[Theorem]{Proposition}
 \newtheorem{Conjecture}[Theorem]{Conjecture}
 \theoremstyle{definition}

\let\to=\rightarrow
 \let\To=\longrightarrow
 \def\Implies{\ifmmode\Longrightarrow \else
         \unskip${}\Longrightarrow{}$\ignorespaces\fi}
 \def\implies{\ifmmode\Rightarrow \else
         \unskip${}\Rightarrow{}$\ignorespaces\fi}
 \def\iff{\ifmmode\Longleftrightarrow \else
         \unskip${}\Longleftrightarrow{}$\ignorespaces\fi}

 \let\:=\colon

\title{On the Rees algebras of $t$-path ideals of cycles}
\author{Oleksandra Gasanova, Jürgen Herzog, Jiawen Shan}
\begin{document}

\begin{abstract}
In this paper we explore certain properties of the Rees algebra of $I_t(C_n)$, the $t$-path ideal of an $n$-cycle. Our main focus is on the cases when such ideals are of fiber type.    
\end{abstract}
\maketitle
\section{Introduction}

Let $G$ be a finite simple graph with the vertex set $V=\{x_1, \ldots, x_n\}$ and let $1\leq t\leq n$ be an integer. A sequence $\{x_{i_1}, x_{i_2},\ldots,x_{i_t}\}$ is called a \textit{$t$-path} of  $G$ if $\{x_{i_1}, x_{i_2}\}, \{x_{i_2}, x_{i_3}\},\ldots, \{x_{i_{t-1}}, x_{i_t}\}$ are distinct edges of $G$. The \textit{$t$-path ideal} $I_t(G)$ associated to $G$ is the squarefree monomial ideal $$I_t(G)=(x_{i_1}\cdots x_{i_t}\mid\{x_{i_1}, \ldots, x_{i_t}\} \textrm{~is~a}~t\textrm{-path~of}~G)$$ in the polynomial ring $S=K[x_1,\ldots,x_n]$.

Note that if we set $t=2$, we recover the definition of the edge ideal, defined and studied by Villareal in \cite{Vil90}. Here it was also proven that for a connected graph $G$ the ideal $I_2(G)$ is of linear type if and only if $G$ is a tree or has a unique cycle of odd length. The notion of a path ideal as defined above was first introduced by Conca and De Negri in \cite{CD99} in 1999 as a generalisation of the notion of an edge ideal. The authors also showed that if $G$ is a tree, then $I_t(G)$ is of linear type and $R(I_t(G))$ is normal and Cohen-Macaulay. 
In \cite{Bd01}  Brumatti and da Silva focused on the case when $G=C_n$ is a cycle, and discussed in which cases the ideal $I_t(C_n)$ is of linear type. They showed that 
\begin{enumerate}
      \item If $t=n-1$, then $I_t(C_n)$ is of linear type;
      \item If $n$ is odd and $t=n-2$ or $t=\frac{n-1}{2}$, then $I_t(C_n)$ is of linear type;
      \item If $gcd(n,t)>1$, then $I_t(C_n)$ is not of linear type;
      \item If $gcd(n,t)=1$
      and $[\frac{n-1}{2}]<t\le n-3$, then $I_t(C_n)$ is not of linear type;
      \item If $gcd(n,t)=1$, $tl\equiv 1 \pmod n$ and $1<t,l\le[\frac{n-1}{2}]$, then $I_t(C_n)$ is not of linear type;
      \item If $gcd(n,t)=1$, $tl\equiv 1 \pmod n$, $1<t\le[\frac{n-1}{2}]$, $[\frac{n-1}{2}]<l<n$ and $n=p(n-l)+a$, $2\leq a<n-l$, then $I_t(C_n)$ is not of linear type.
\end{enumerate}
The authors also provided references for presence of linear type in a few other cases for instance, $t=1$, $t=n$, and $t=2$ when $n$ is odd.

The next best possible situation after linear type is so-called \emph{fiber type}. Villarreal proved in \cite[Theorem~8.2.1]{Vil01} that $I_2(G)$ is of fiber type.
In \cite{HHV05} it was shown that all polymatroidal ideals are of fiber type.
This allows to prove that for instance $t$-path ideals of complete graphs are of fiber type, since they are squarefree Veronese ideals and hence polymatroidal.

In this paper we are mainly focusing on path ideals of cycles $C_n$, and study which $I_t(C_n)$ are of fiber type, but not of linear type. 
The paper is organized as follows. In Section~2 we give the necessary background: we recall the definitions of the Rees algebra, the symmetric algebra and the fiber cone of an ideal, and explain the notions of linear type and fiber type. In Section~3 we discuss the dimension and defining relations of the fiber cone $F(I_t(C_n))$. In Section~4 we compute a Gröbner basis of the defining ideal of $R(I_{n-2}(C_n))$. It allows us to conclude that $I_{n-2}(C_n)$ is of fiber type, and to compute the Hilbert series and Cohen-Macaulay type of $R(I_{n-2}(C_n))$. In Section~5 we compute a Gröbner basis of the defining ideal of $R(I_{\frac{n}{2}}(C_n))$ when $n$ is even, and conclude that $I_{\frac{n}{2}}(C_n)$ is also of fiber type. Section~6 contains a summary or results regarding linear and fiber type of $I_t(C_n)$, and a few open questions.

\section{Preliminaries and background}

In this section we provide some algebraic background and notation which will be used throughout this paper.

Let $I=(u_1,\ldots, u_m)$ be an equigenerated graded ideal of the polynomial ring $S=K[x_1,\ldots,x_n]$ over the field $K$. Let $s$ be a variable over $S$. We denote by $R(I)=S[u_1s,\ldots,u_ms]$ the Rees algebra of $I$. Then $R(I)$ is bigraded with $\deg x_i =(1,0)$ for $i=1, \ldots, n$ \and $\deg u_j s =(0,1)$ for $j=1,\ldots, m$.

Let $T=S[y_1,\ldots, y_m]$ be the polynomial ring over $S$ in the variables $y_1,\ldots, y_m$. We define a bigrading on $T$ by setting $\deg x_i=(1,0)$ for $i=1, \ldots, n$ and $\deg y_j=(0,1)$ for $j=1,\ldots, m$. Then the natural surjective bigraded $K$-algebra homomorphism $$\varphi: T\longrightarrow R(I)$$ with $\varphi(x_i)=x_i$ for $i=1,\ldots, n$ and $\varphi(y_j)=u_j s$ for $j=1,\ldots, m$ is bigraded.

Let $J=\Ker(\varphi)$ be the defining ideal of $R(I)$. One says that $I$ satisfies the \textit {$x$-condition} if there exists a monomial order $<$ on $T$ such that for any minimal monomial generator $u$  of $ \ini_<(J)$ we have $\deg_{x_i} u \leq 1$ for all $i$. The $x$-condition guarantees that all powers of $I$ have linear resolution (see \cite[Corollary~10.1.7]{HH}). Moreover, if $<$ is a product order of an arbitrary $<^{\#}$ on $B_1=\{y_1,\ldots, y_m\}$, and $<_{lex}$ induced by $x_1>x_2>\cdots>x_n$ on $B_2=\{x_1,\ldots, x_n\}$, then all powers of $I$ have linear quotients (see \cite[Theorem~10.1.9]{HH}).

Let $\alpha=(u_{ij})_{m\times r} $ be a relation matrix of $I$. Then for $j=1,\ldots, r$, the polynomials $\ell_j=\sum_{i=1}^{m}u_{ij}y_i\in \Ker(\varphi)$.  
Let $L=(\ell_1,\ldots, \ell_r)$. Then $T/L$ is isomorphic to the symmetric algebra $S(I)$. Since $L\subseteq J$, there exists a surjective bigraded $K$-algebra homomorphism $\psi: S(I)\longrightarrow R(I)$. 

The ideal $I$ is called of \emph{ linear type}, if $\psi$ is an isomorphism, that is, if $\Ker(\varphi)=L$.


In a sense the next best case after linear type is fiber type, which is defined as follows. Let $\mm=(x_1,\ldots, x_n)$ be the graded maximal ideal of $S$. The $K$-algebra $F(I)=R(I)/\mm R(I)$ is called the \emph{fiber cone} of $I$. If $I$ is equigenerated, then $F(I)\iso K[u_1,\ldots, u_m]$. 

The map $\varphi: T\longrightarrow R(I)$ induces a surjective $K$-algebra homomorphism $\varphi': T/\mm T \longrightarrow R(I)/\mm 
R(I).$ Since $T/\mm T\cong K[y_1,\ldots, y_m]$, the map $\varphi'$ may be identified with $$K[y_1,\ldots, y_m] \longrightarrow K[u_1,\ldots, u_m], \quad y_i\mapsto u_i.$$

The elements in $H=\Ker(\varphi')$ are called \emph{the fiber relations} of $I$. They are polynomials in $R=K[y_1,\ldots,y_m]$.

 Obviously, $HT\subseteq J$ and therefore, $L+HT\subseteq J$. The ideal is called of \emph{fiber type}, if $J=L+ HT$. Note that 
$
\text{linear type} \Rightarrow \text{fiber type} \Rightarrow \text{general case},
$
and none of the implications can be reversed.

In this paper we are mainly focusing on path ideals of cycles $C_n$, and study which $I_t(C_n)$ are of fiber type, but not of linear type. Throughout the paper, we consider $1<t<n-1$,  since the case $t=n$ is trivial, and $I_t(C_n)$ are of linear type when $t=1$ and $t=n-1$. We have $m=n$ since the number of $t$-paths equals the number of vertices, and the homomorphism $\varphi$ defined above will map $y_j$ to $x_j\cdots x_{j+t-1}s$. Since we are working with an $n$-cycle, the indices of all the variables (both $x_i$ and $y_j$) will be taken modulo $n$. In particular, we will interchangeably use $x_0$ and $x_n$ ($y_0$ and $y_n$) to mean the same thing.

\section{Dimension and relations of $F(I_t(C_n))$}
 In this section we focus on some properties of the fiber cone of the path ideals of cycles.

\begin{Theorem}\label{gcd}
Let $d=\gcd(n,t)$. Then $\dim F(I_t(C_n))=n-d+1$. In particular, $F(I_t(C_n))$ is a polynomial ring if and only if $\gcd(n,t)=1$.
\end{Theorem}
\begin{proof}
 
It is known that the dimension of a toric ring equals the rank of the matrix whose column vectors are the exponent vectors of the monomial generators of the algebra. In our case, 
 
$$A=\begin{pmatrix}
c_0 & c_{n-1} & \ldots & c_2 & c_1\\
c_1 & c_0 & \ldots & c_3 & c_2\\
\vdots & \vdots & ~~ & \vdots & \vdots & \\
c_{n-2} & c_{n-3} & \ldots & c_0 & c_{n-1}\\
c_{n-1} & c_{n-2} & \ldots
& c_1 & c_0
\end{pmatrix},$$
where $c_0=c_1=\cdots=c_{t-1}=1$ and $c_{t}=c_{t+1}=\cdots=c_{n-1}=0$.
This is a so-called \textit {circulant} matrix, and by \cite{circulant} its rank equals $n-\deg \gcd(f(x), x^n-1)$, where $f(x)=c_0+c_1x+\cdots+c_{n-1}x^{n-1}$. In our case, $$f(x)=1+x+\cdots+x^{t-1}=\prod_{\substack{\xi^t=1\\
\xi\not=1}}(x-\xi)$$ and

$$x^n-1=\prod_{{\xi^n=1}}(x-\xi).$$
Therefore, our goal is to show that these polynomials have exactly $d-1$ roots in common. Equivalently, our goal is to show that $x^t-1$ and $x^n-1$ have exactly $d$ roots in common. On the one hand, every $d$th root of $1$ is also a $n$th and a $t$th root of $1$, so we have at least $d$ common roots. On the other hand, if $\xi$ is both a $n$th and a $t$th root of $1$, it has to be a $d$th root of $1$. Indeed, $d=k_1n+k_2t$ for some integers $k_1,k_2$, which then implies $\xi^d=\xi^{k_1n+k_2t}=\xi^{k_1n}\xi^{k_2t}=1$. Therefore, we can not have more than $d$ common roots. We conclude that the two polynomials have exactly $d$ common roots, as desired.
\end{proof}

\begin{Theorem}
\label{fiber}
Consider the $K$-algebra homomorphism $\varphi': R=K[y_1,\ldots,y_n]\To F(I_t(C_n))$ given by $\varphi'(y_i)= x_ix_{i+1}\cdots x_{i+t-1}$. Let $d=\gcd(n,t)$ and let $$m_i=\prod_{\substack{1\le j\le n\\j\equiv i\pmod d}}y_j$$
for $i=1,\ldots, d$. Then the defining ideal $H$ of $F(I_t(C_n))$  coincides with the ideal $H'=(m_1-m_d,m_2-m_d,\ldots, m_{d-1}-m_d)$.

\end{Theorem}
\begin{proof}
We may assume that $d>1$, otherwise the theorem follows directly from \Cref{gcd}. Let $n=n'd$ and $t=t'd$. Note that $\varphi'(m_i)=(x_1\cdots x_n)^{t'}$ for all $i$. 
We will prove this for $i=1$, the proof for the other $i$ is analogous.
Let 
$$v_1=x_1\cdots x_{d}, v_2=x_{d+1}\cdots x_{2d},\ldots,v_{n'}=x_{n-d+1}\cdots x_n.$$ In other words, $v_j$ encodes a $d$-path ending at $dj$ for all $j=1,\ldots, n'$.
Then we have
\begin{align*}
\varphi'(y_1)&=v_1v_2\cdots v_{t'-1}v_{t'},\\
\varphi'(y_{d+1})&=v_2v_3\cdots v_{t'}v_{t'+1},\\
\vdots\\
\varphi'(y_{n-2d+1})&=v_{n'-1}v_{n'}\cdots v_{t'-3}v_{t'-2},\\
\varphi'(y_{n-d+1})&=v_{n'}v_{1}\cdots v_{t'-2}v_{t'-1}.
\end{align*}
Therefore, $\varphi'(m_1)$ is the product of all of these expressions. We can see that the product of each column equals $v_1v_2\cdots v_{n'}=x_1\cdots x_n$, and there are $t'$ columns in total, which implies $\varphi'(m_1)=(x_1\cdots x_n)^{t'}$. This shows  that $m_i-m_d\in H$ for $i=1,\ldots,d-1$. Therefore, $H'\subseteq H$. \Cref{gcd} implies that $\height H=d-1$, therefore, it is enough to show that $H'$ is a prime ideal of height $d-1$. The initial terms of the set $\mathcal{G}  =\{m_1-m_d,m_2-m_d,\ldots, m_{d-1}-m_d\}$   with respect to the lexicographical order $<$ are the monomials $m_1,\ldots, m_{d-1}$. Since these monomials are pairwise coprime, it follows that $\mathcal{G}$ is a Gr\"obner basis of $H'$, and hence  $\ini_{<}(H')=(m_1,\ldots, m_{d-1})$. Therefore, $\height H'=\height \ini_{<}(H')=d-1$. From this Gröbner basis we also see that all variables dividing $m_d$ are non-zerodivisors modulo $\ini_{<}(H')$, and therefore modulo $H'$. Since $H'=(m_i-m_j\mid 1\le i<j\le n)$, we obtain by symmetry that all $y_i$ are non-zerodivisors modulo $H'$.

Let $y=y_d\cdots y_{n}$. Then $y$ is a non-zerodivisor modulo $H'$. 
Therefore, the natural map $R/H'\To R_y/H'R_y$ is injective. We will show that $R_{y}/H'R_{y}$ is a domain. This will then imply that $R/H'$ is a domain, as desired. We know that $R_y=K[y_1,\ldots, y_{d-1},y_d^{\pm1},\ldots,y_n^{\pm1}]$ and that $H'R_y$ is generated by all the $m_i-m_d=y_iy_{d+i}\cdots y_{n-d+i}-y_dy_{2d}\cdots y_{n}$, $i=1,\ldots, d-1$. Equivalently, we can say that $H'R_y$ is generated by all the $y_i-(y_dy_{2d}\cdots y_{n}) \cdot (y_{d+i}\cdots y_{n-d+i})^{-1}$, $i=1,\ldots, d-1$. This implies that $R_y/H'R_y\iso K[y_d^{\pm 1}, y_{d+1}^{\pm 1},\ldots, y_n^{\pm 1}]$, which finishes the proof.
\end{proof}

\section{The structure of the Rees ring of $I_{n-2}(C_n)$}
\label{sec:n-2}
Let $n\geq 3$ be a integer and let $L$ be the defining ideal of the symmetric algebra $S(I_{n-2}(C_n))$.
By \cite[Proposition~2.5]{Bd01}  we know that $L=(f_1,\ldots, f_{n-1},g_1)$, where $f_j={x_{j-2}y_j}-x_jy_{j+1}$ for $j=1,\ldots,n-1$ and $g_1=x_0y_1-x_{n-2}y_0$. Note that if we extend the definition of $f_j$ to $j=n$, we will obtain $g_1=-f_n$. However, for reasons that will become clear later, it is not convenient to make such an extension for Gröbner bases computations. For all other purposes we may say $L=(f_1,\ldots, f_n)$.
As a preparation for a computation of a Gröbner basis for $L$ (and then later for $J$, the defining ideal of $R(I_{n-2}(C_n))$), we extend the set of its generators.

\begin{Proposition}
\label{linear relations}
$L=(f_1,\ldots,f_{n-1},g_1,\ldots,g_{\lfloor\frac{n}{2}\rfloor}),$
where 
$$g_k={x_{2k-2}\prod_{\substack{i\in [0,2k)\\i \text{ odd}}}y_i}-x_{n-2}\prod_{\substack{i\in [0,2k)\\i \text{ even}}}y_i~~for ~~k=1,\ldots, \lfloor \frac{n}{2}\rfloor.$$
\end{Proposition}
\begin{proof}

  We will show by induction that $g_k\in L= (f_1,\ldots,f_{n-1},g_1)$ for all $k=1,\ldots, \lfloor\frac{n}{2}\rfloor$. The base case holds by definition. Now, modulo $L$ we have:
 \begin{align*}    
&x_{n-2}\prod_{\substack{i\in [0,2k)\\i \text{ even}}}y_i=y_{2k-2}\left(x_{n-2}\prod_{\substack{i\in [0,2k-2)\\i \text{ even}}}y_i\right)
=y_{2k-2}\left(x_{2k-4}\prod_{\substack{i\in [0,2k-2)\\i \text{ odd}}}y_i\right)\\
=&x_{2k-2}y_{2k-1}\prod_{\substack{i\in [0,2k-2)\\i \text{ odd}}}y_i
=x_{2k-2}\prod_{\substack{i\in [0,2k)\\i \text{ odd}}}y_i.
\end{align*}
Here the second equality follows from the induction hypothesis and the third equality comes from $x_{2k-4}y_{2k-2}-x_{2k-2}y_{2k-1}\in L$. 
\end{proof}

\begin{Theorem}\label{GB1}
The generators of $L$ given in \Cref{linear relations} form a Gröbner basis of $L$ with respect to any monomial order $<$ on $T$ which is the product order of $<_1$ and $<_2$, where $<_1$ is the reverse lexicographic order on $R$ induced by $y_1>y_2>\cdots > y_{n-1}>y_0$, and $<_2$ is an arbitrary monomial order on $S$.
\end{Theorem}
The proof of this fact can be found in the appendix. There, and further in this paper, the leading terms of binomials are underlined for convenience whenever a Gröbner basis or an initial ideal is being computed. 

\medskip

Recall that by \Cref{fiber} we have $H=0$ if $n$ is odd, and if $n$ is even we have $H=(h)$, where $$h=\prod_{\substack{i\in [0,n)\\i \text{ odd}}}y_i-\prod_{\substack{i\in [0,n)\\i \text{ even}}}y_i.$$ 
We will now describe a Gröbner basis of $L+HT$, which, if $n$ is even,
 will later turn out to be equal to $J$, the defining ideal of $R(I_{n-2}(C_n))$.

\begin{Corollary}
\label{GBfibers}
Let $n$ be even. 
Then the polynomials $\{f_1,\ldots,f_{n-1},g_1,\ldots,g_{\frac{n}{2}},h\}$ form a Gröbner basis for $L+HT$ with respect to any monomial order given in \Cref{GB1}.
\end{Corollary}
The proof of this fact can be found in the appendix.

We are now ready to describe the defining ideal $J$ of $R(I_{n-2}(C_n))$.

\begin{Theorem}
\label{Rees}
 Let as before $L=(f_1,\ldots, f_n)$, where  $f_j=x_{n-2+j}y_j-x_jy_{j+1}$ for $j=1,\ldots, n$. As before, let $H=(h)$, where $h=\prod_{i=0}^{\frac{n}{2}-1}y_{2i}- \prod_{i=0}^{\frac{n}{2}-1}y_{2i+1}$. Then $J=L$ if $n$ is odd, and $J=L+HT$ if $n$ is even. 
\end{Theorem}

\begin{proof}
The proof is known when $n$ is odd, see \cite[Theorem~3.4]{Bd01}. So we only have to deal with the case when $n$ is even. We set $J'=(f_1,\ldots,f_n, h)$. Then obviously $J'\subseteq J$.

 We have $\dim R(I_{n-2}(C_n))=\dim S+1=n+1$. Thus  $J'=J$, once we have shown that $J'$ is a prime ideal and that $\dim T/J'=n+1$.

It can be seen from the Gr\"obner basis of $J'$ given in \Cref{GBfibers}  that $y_0$ is a non-zerodivisor modulo  $J'$. By rotational symmetry it follows that all $y_i$ are non-zerodivisors modulo $J'$. This implies that  $T/J'\To T_y/J'T_y$ is injective, where $y=y_0y_1\cdots y_{n-1}$, and $J'T_y =
(\{x_i-x_{i-2}w_i\mid i=1,\ldots,n-1\},h)$ with  $w_i=y_iy_{i+1}^{-1}$. In particular,  $x_1-x_{n-1}w_1$ and  $x_3-x_1w_3$ belong to $J'T_y$. It follows that $x_3-x_1w_3+w_3(x_1-x_{n-1}w_1)=x_3-x_{n-1}w_1w_3\in J'T_y$, and we may replace $x_3-x_1w_3$  by  $x_3-x_{n-1}w_1w_3$ in the set of generators of $J'T_y$, since $x_3-x_1w_3$ can also be written as a  linear combination of $x_1-x_{n-1}w_1$ and $x_3-x_{n-1}w_1w_3$. 
By applying  similar substitutions, we may replace 
$x_i-x_{i-2}w_i$ by $x_i-x_{n-1}\prod_{j=0}^{\frac{i-1}{2}}w_{2j+1}$, if $i$ is odd, and we may replace  it by $x_i-x_{0}\prod_{j=0}^{\frac{i-2}{2}}w_{2j+2}$, if $i$ is even. In particular, $x_{n-1}-x_{n-3}w_{n-1}$ can be replaced by $x_{n-1}-x_{n-1}\prod_{j=0}^{\frac{n-2}{2}}w_{2j+1}$, which after multiplication with the unit  $\prod_{j=0}^{\frac{n}{2}-1}y_{2j}$, can be replaced by $x_{n-1}h$. Similarly, $x_{n}-x_{n-2}w_{n}$ can be replaced by $x_nh$.

We conclude that 
$T_y/J'T_y\iso (R[x_0,x_{n-1}]/(h)R[x_0,x_{n-1}])_y
$. It follows  that  $T_y/J'T_y$ is a domain, since $h$ is an irreducible polynomial in the polynomial ring $R[x_0,x_{n-1}]$. This proves  that $J'$ is a prime ideal.

Finally, notice that 
\begin{align*}
    \dim T/J'=\dim T_y/J'T_y=\dim (R[x_0,x_{n-1}]/(h)R)_y\\=\dim R[x_0,x_{n-1}]/(h)R[x_0,x_{n-1}]=n+1, 
\end{align*}
as desired.

\end{proof}

The preceding theorem together with \Cref{GBfibers} yields  the following structural results about the Rees ring of $I_{n-2}(C_n)$. 

\begin{Corollary}
    \label{niceinitial} 
 Let $n\geq 3$ and let $I=I_{n-2}(C_n)$. Then the following holds:
 \begin{enumerate}
\item[{\em (a)}] $R(I)$ is a normal Cohen--Macaulay domain. 
\item[{\em (b)}] $I$  is of linear  type if $n$ is odd, and $I$  is of fiber type if $n$ is even.  
\item[{\em (c)}] $I^s$ has linear quotients  for all $s\geq 1.$
 \end{enumerate}
 \end{Corollary}

\begin{proof}
(a)  Let $J$ as before be the defining ideal of $R(I)$. By \Cref{GB1} and \Cref{GBfibers},  $J$ has a squarefree initial ideal. By \cite[Proposition~13.15]{St} this implies that $R(I)$ is a normal. Then, since $R(I)$ is a toric ring, we obtain  that $R(I)$ is a  Cohen--Macaulay domain,  see \cite[Theorem~1]{Ho}. 

(b) This can be seen from the generators of $J$ given in \Cref{Rees}.

(c) It is shown in \Cref{GB1} and \Cref{GBfibers} that $J$ satisfies the $x$-condition. If in \Cref{GB1} we set $<_2=<_{lex}$ to be the lexicographic order induced by $x_1>x_2>\cdots>x_{n-1}>x_0$, then the corresponding product order $<$ satisfies the assumptions of \cite[Theorem~10.1.9]{HH}. This yields the desired conclusion. 
\end{proof}

In the following we give an interesting interpretation of the fiber relation $h$ of the ideal $I=I_{n-2}(C_n)$ when $n$ is even. With the notation as in \Cref{Rees} we set
$L=(f_1,\ldots,f_n)$. Then the symmetric algebra $S(I)$ of $I$ has the presentation $T/L$.  Exchanging the role of the $x_i$ and the $y_j$, we may view $T/L$ also as the symmetric algebra of an $R$-module $E$, which is called the Jacobian dual of $I$. The relation matrix $A$ of $E$  is defined by the equation $\fb=A\xb$, where $\fb=[f_1,\ldots,f_n]^T$ and $\xb=[x_1,\ldots,x_n]^T$. It can be seen that $A$ is a skew-symmetric matrix and that the Pfaffian of $A$ is just the fiber relation $h$.

\medskip We are ready to compute the Hilbert series and Cohen-Macaulay type of $R(I_{n-2}(C_n))$.
\begin{Theorem}\label{Hilbert series}
Let $I=I_{n-2}(C_n)$. Then:
\begin{enumerate}

\item [\emph{(a)}] if $n=2s+1$, we have 
    $$\HS(R(I))=\frac{\sum_{k=0}^{s-1}{(1+z)^{2k+1}z^{s-1-k}}+z^s}{(1-z)^{2s+2}},$$
\item [\emph{(b)}] if $n=2s$, we have
$$
 \HS(R(I))=\frac{\sum_{k=0}^{s-1}{(1+z)^{2k}z^{s-1-k}}}{(1-z)^{2s+1}}
$$
\end{enumerate}

    
\end{Theorem}
\begin{proof}
As before, let $f_1=\underline{x_{n-1}y_1}-x_1y_2$, $f_2=\underline{x_0y_2}-x_2y_3$, $\ldots,$ $f_{n-1}=\underline{x_{n-3}y_{n-1}}-x_{n-1}y_0$ and let 
$$g_k=\underline{x_{2k-2}\prod_{\substack{i\in [0,2k)\\i \text{ odd}}}y_i}-x_{n-2}\prod_{\substack{i\in [0,2k)\\i \text{ even}}}y_i$$
for $k=1,\ldots, \lfloor \frac{n}{2}\rfloor$. Recall from \Cref{GB1} and \Cref{linear relations} that the polynomials $f_1,\ldots,f_{n-1},g_1,\ldots, g_{\lfloor \frac{n}{2}\rfloor}$ form a Gröbner basis of $L=(f_1,\ldots, f_{n-1},g_1)$. If $n=2s+1$, we have $J=L$ by \cite[Theorem 3.4]{Bd01}. Also recall from \Cref{Rees} that if $n$ is even, we have $J=L+hT$, where $h=\underline{\prod_{\substack{i\in [0,n)\\i \text{ odd}}}y_i}-\prod_{\substack{i\in [0,n)\\i \text{ even}}}y_i$, and $f_1,\ldots,f_{n-1},g_1,\ldots,g_{\frac{n}{2}},h$ form a Gröbner basis of $J$.

Let $K_{n}$ denote the initial ideal of $J$. Then if $n=2s+1$ we have
\begin{multline*}
K_{2s+1}=(x_{2s}y_1,x_0y_2,\ldots,x_{2s-2}y_{2s},x_0y_1,x_2y_1y_3,\ldots,\\  x_{2s-4}y_1y_3\cdots y_{2s-3},x_{2s-2}y_1y_3\cdots y_{2s-1}).
\end{multline*}
Define an auxiliary ideal
\begin{multline*}
 K'_{2s+1}=K_{2s+1}+(y_1y_3\cdots y_{2s-1})
 =(x_{2s}y_1,x_0y_2,\ldots,x_{2s-2}y_{2s},x_0y_1,x_2y_1y_3,\ldots,\\
 x_{2s-4}y_1y_3\cdots y_{2s-3},y_1y_3\cdots y_{2s-1}).   
\end{multline*}
Now, if $n=2s$, the corresponding initial ideal is
\begin{multline*}
    K_{2s}=(x_{2s-1}y_1,x_0y_2,\ldots,x_{2s-3}y_{2s-1},x_0y_1,x_2y_1y_3,\ldots, \\ 
    x_{2s-4}y_1y_3\cdots y_{2s-3},y_1y_3\cdots y_{2s-1}).
\end{multline*}

Note that our Gröbner basis is not reduced and the leading term of $h$ divides the leading term of $g_s$.
From this point on, the proof will be divided into two steps:

Step~1: We will show that $(a)$ is equivalent to 
$$
\HS(T/K'_{2s+1})=\frac{\sum_{k=0}^{s-1}{(1+z)^{2k+1}z^{s-1-k}}}{(1-z)^{2s+2}}.
$$
Indeed, we have $K_{2s+1}+(y_1y_3\cdots y_{2s-1})=K'_{2s+1}$ and $K_{2s+1}:(y_1y_3\cdots y_{2s-1})=(x_{2s},x_1,x_3,\ldots,x_{2s-3},x_0,x_2,\ldots,x_{2s-2})=(x_{2s},x_0,x_1,\ldots,x_{2s-2})$. Therefore,

\begin{align*}
    &\HS(T/K_{2s+1})\\&=\HS(T/(K_{2s+1}+(y_1y_3\cdots y_{2s-1})))+ z^s\HS(T/(K_{2s+1}:(y_1y_3\cdots y_{2s-1})))\\&=\HS(T/K'_{2s+1})+\frac{z^s}{{(1-z)}^{2s+2}}.
\end{align*}

Step~2: We will prove the theorem by induction on $n$. There will be two induction steps, one transitioning from $n$ even to $n$ odd, and the other one transitioning from $n$ odd to $n$ even.

    (1) Let us assume we know the  theorem is true for all $n\le 2s$. We want to prove it for $n=2s+1$. 
    
    We are  working inside the ring $T=K[x_0,\ldots,x_{2s},y_{0},\ldots,y_{2s}]$. The factor ring $T/K_{2s}T$ has  the Hilbert series given by (b), divided by $(1-z)^2$ (since we have $2$ extra variables $x_{2s}$ and $y_{2s}$).
    Let $\tilde{K}_{2s}\subset T$ be the ideal obtained from $K_{2s}T$ by switching $x_{2s-1}\leftrightarrow x_{2s}$. Clearly, $T/\tilde{K}_{2s}$ has the same Hilbert series as $T/K_{2s}T$ (which we know by the induction hypothesis). Now, 
    \begin{multline*}
     \tilde{K}_{2s}+(x_{2s-2}y_{2s})=(x_{2s}y_1,x_0y_2,\ldots,x_{2s-3}y_{2s-1},x_{2s-2}y_{2s} x_0y_1,x_2y_1y_3,\\ \ldots, x_{2s-4}y_1y_3\cdots y_{2s-3},y_1y_3\cdots y_{2s-1})=K'_{2s+1}   
    \end{multline*}
and $\tilde{K}_{2s}:(x_{2s-2}y_{2s})=\tilde{K}_{2s}$. Therefore 
$$\HS(T/\tilde{K}_{2s})=\HS(T/K'_{2s+1})+z^2\HS(T/\tilde{K}_{2s})
$$
and hence 
\begin{align*}
    \HS(T/K'_{2s+1})&=(1-z^2)\HS(T/\tilde{K}_{2s}t)=(1-z^2)\frac{\sum_{k=0}^{s-1}{(1+z)^{2k}z^{s-1-k}}}{(1-z)^{2s+3}}\\
    &=\frac{\sum_{k=0}^{s-1}{(1+z)^{2k+1}z^{s-1-k}}}{(1-z)^{2s+2}},
\end{align*}
which by Step~1 is equivalent to having the desired Hilbert series for $T/K_{2s+1}$.
\medskip

(2) Let us assume we know the statement is true for all $n\le 2s+1$, we want to prove it for $n=2s+2$. 

Now we are working inside the ring $T=K[x_0,\ldots,x_{2s+1},y_{0},\ldots,y_{2s+1}]$. 
   The factor ring $T/K_{2s+1}T$ has the Hilbert series given by (a), divided by $(1-z)^2$ (since we have extra variables $x_{2s+1}$ and $y_{2s+1}$).
    Let $\tilde{K}_{2s+2}\subset T$ be the ideal obtained from $K_{2s+2}$ by switching $x_{2s}\leftrightarrow x_{2s+1}$. Clearly, $T/\tilde{K}_{2s+2}$ has the same Hilbert series as $T/K_{2s+2}$ (which we do not know yet, but will compute). Now, 
    \begin{multline*}
     \tilde{K}_{2s+2}+(y_{2s+1})=(x_{2s}y_1,x_0y_2,\ldots,x_{2s-3}y_{2s-1},x_{2s-2}y_{2s}, y_{2s+1},\\ x_0y_1,x_2y_1y_3,\ldots, x_{2s-4}y_1y_3\cdots y_{2s-3},y_1y_3\cdots y_{2s-1})=K_{2s+1}+(y_{2s+1}).   
    \end{multline*}
and 
\begin{multline*}
\tilde{K}_{2s+2}:(y_{2s+1})=(x_{2s}y_1,x_0y_2,\ldots,x_{2s-3}y_{2s-1},x_{2s-2}y_{2s}, x_{2s-1},\\
x_0y_1,x_2y_1y_3,\ldots, x_{2s-4}y_1y_3\cdots y_{2s-3},y_1y_3\cdots y_{2s-1})=K'_{2s+1}+(x_{2s-1}). 
\end{multline*}
By the induction hypothesis we know the Hilbert series of $T/K_{2s+1}T$ and hence by Step~1 the Hilbert series of $T/K'_{2s+1}T$. Therefore, we also know the Hilbert series of  $T/(K'_{2s+1}+(x_{2s-1}))T$ since $x_{2s-1}$ is a non-zerodivisor modulo $K'_{2s+1}T$.
Then we obtain
\begin{align*}
  &\HS(T/K_{2s+2})=\HS(T/\tilde{K}_{2s+2})\\
  &=\HS(T/(K_{2s+1}+(y_{2s+1}))T)+z\HS(T/(K'_{2s+1}+(x_{2s-1}))T)\\
  &=\frac{\sum_{k=0}^{s-1}{(1+z)^{2k+1}z^{s-1-k}}+z^s}{(1-z)^{2s+3}}+\frac{z\sum_{k=0}^{s-1}{(1+z)^{2k+1}z^{s-1-k}}}{(1-z)^{2s+3}}\\
  &=\frac{z^s+(1+z)\sum_{k=0}^{s-1}{(1+z)^{2k+1}z^{s-1-k}}}{(1-z)^{2s+3}}=\frac{z^s+\sum_{k=0}^{s-1}{(1+z)^{2(k+1)}z^{s-(k+1)}}}{(1-z)^{2s+3}}\\
  &=\frac{z^s+\sum_{k=1}^{s}{(1+z)^{2k}z^{s-k}}}{(1-z)^{2s+3}}=\frac{\sum_{k=0}^{s}{(1+z)^{2k}z^{s-k}}}{(1-z)^{2s+3}}.
\end{align*}

\end{proof}

Let $I=I_{n-2}(C_n)$. 
Since $R(I)$ is a domain and its  $h$-vector is symmetric, it follows from \Cref{Hilbert series} and \cite[Theorem~4.4]{Stanley} that $R(I)$ is Gorenstein when $n$ is even. \Cref{Hilbert series} also implies that $R(I)$ is not Gorenstein, if $n$ is odd. We actually have the following

\begin{Theorem}
 \label{typetwo}  
 Let $n>1$ be an odd integer and let $I=I_{n-2}(C_n)$.  Then the Cohen-Macaulay type  of $R(I)$ is $2$.
 \end{Theorem}

 \begin{proof} The Rees algebra $R(I)$ modulo the image of the sequence $$\sb=x_1-y_3,x_2-y_4, \ldots, x_{i-2}-y_i, \ldots, x_{n-2}-y_n, x_{n-1}-y_1, x_n, y_2$$ in $R(I)$  is isomorphic to $K[x_1,\ldots,x_{n-1}]/\bar{J}$, where 
 \[
 \bar{J}=(x_1^2-x_2x_3,x_2^2-x_3x_4,\ldots,x_{n-3}^2-x_{n-2}x_{n-1},x_{n-2}^2,x_{n-1}^2, x_1x_2).
 \]
We denote by $<$ the lexicographical order on $K[x_1,\ldots,x_{n-1}]$ induced by $x_1> x_2>\cdots>x_{n-1}$. Then $\ini_{<}(\bar{J})$ contains $x_1^2,x_2^2, \ldots, x_{n-1}^2$,  This implies that $\dim R(I)/(\sb)=0$.  By \Cref{niceinitial}(a)  the ring $R(I)$ is Cohen-Macaulay. Since $\dim R(I)$ coincides with the length of the sequence $\sb$, it follows therefore from \cite[Theorem 2.1.2(c)]{BH} that $\sb$ is an  $R(I)$-sequence. From this we deduce that $R(I)$ and $\bar{R}=R(I)/(\sb)$ have the same Cohen-Macaulay type, see \cite[Proposition A.6.2]{HH}. Thus is suffices to show that the Cohen-Macaulay type of $\bar{R}$ is $2$. 

Let $A=(x_1^2-x_2x_3,\ldots,x_{n-3}^2-x_{n-2}x_{n-1},x_{n-2}^2,x_{n-1}^2)\subset K[x_1,\ldots, x_{n-1}]$. Then $\bar{J}=A+(x_1x_2)$, and $A$ is generated by a regular sequence since $\ini_{<}(A)=(x_1^2,\ldots,x_{n-1}^2)$ is generated by a regular sequence. Given this, it is well-known that the canonical module $\omega_{\bar{R}}$ is isomorphic to $(A:(x_1x_2))/A$. For the convenience of the reader we show why this is the case. 

The ring $\bar{S}=K[x_1,\ldots,x_{n-1}]/A$  is a  complete intersection and hence Gorenstein. In particular, up to a shift in the grading, we have that $\omega_{\bar{S}}\iso \bar{S}$.  Then it follows from \cite[Theorem 3.3.7]{BH} that, up to a shift, 
 \[
 \omega_{\bar{R}} \iso \Hom_{\bar{S}}(\bar{R},\bar{S})\iso 0:_{\bar{S}}(x_1x_2)\bar{S}. 
 \]
 The last equation results from the fact that $\bar{R}=\bar{S}/(x_1x_2)\bar{S}$. 

By \cite[Proposition~3.3.11]{BH}, the minimal number of generators of $\omega_{\bar{R}}$ is the Cohen-Macaulay type of $\bar{R}$.  Thus it remains to be shown that  $(A:(x_1x_2))/A$ is generated by $2$ elements.
 
In fact, we will show that $A:(x_1x_2)=A+(x_1,x_2x_4\ldots x_{n-3}x_{n-1})$. In other words, we want to determine for which polynomials $p$ we have $px_1x_2 \in A$. The inclusion $A+(x_1,x_2x_4\ldots x_{n-3}x_{n-1})\subseteq A:(x_1x_2)$ is obvious, so we will focus on the other one. 
Since $\ini_{<}(A)=(x_1^2,\ldots,x_{n-1}^2)$, the squarefree monomials of $K[x_1,\ldots,x_{n-1}]$ form a $K$-basis of $K[x_1,\ldots,x_{n-1}]/A$. 
One way of proving that an element belongs to $A$ is to prove that it reduces to $0$ by the division algorithm with respect to a Gröbner basis of $A$. In this case the reduction algorithm can be described as substituting $x_1^{2} \to x_2x_3,\ldots,x_{n-3}^{2} \to x_{n-2}x_{n-1}, x_{n-2}^{2} \to 0, x_{n-1}^2 \to 0$. Let $\bar{p}$ denote the remainder of $p$ with respect to this Gröbner basis. Since $p\equiv\bar{p} \mod A$, we may assume without loss of generality that $p=\bar{p}$. In other words, we may assume that all terms of $p$ are squarefree. Now we need to apply the reduction algorithm to $x_1x_2p$. In the first step of the proof we will show by induction the following claim: for every pair of squarefree monomials $m_1\not=m_2$ such that $\overline{x_1x_2m_1}=\overline{x_1x_2m_2}$ we have $\overline{x_1x_2m_1}=\overline{x_1x_2m_2}=0$. This will imply that each monomial in the support of $x_1x_2p$ has to reduce to $0$ individually, which will reduce our problem to the case where $p$ is a monomial. In the second step of the proof we will show by induction the following claim: if $m$ is a squarefree monomial such that $\overline{x_1x_2m}=0$, then $m\in A+(x_1,x_2x_4\ldots x_{n-3}x_{n-1})$.
 Note that if $n=3$, we have $A=(x_1^2,x_2^2)$ and $A:(x_1x_2)=A+(x_1,x_2)=(x_1,x_2)$. Both claims hold trivially in this case, and this settles the base case for inductive proofs of both claims.

To prove the first claim by induction, assume that the exists a pair of monomials $m_1=x_1^{\alpha_1}\cdots x_{n-1}^{\alpha_{n-1}}$ and $m_2=x_1^{\beta_1}\cdots x_{n-1}^{\beta_{n-1}}$ with each $\alpha_i$ and $\beta_i$ either $0$ or $1$ such that $m_1\not=m_2$ and $\overline{x_1x_2m_1}=\overline{x_1x_2m_2}\not=0$. First of all note that none of the $m_i$ is divisible by $x_1$, in other words, $\boxed{\alpha_1=\beta_1=0}$. Indeed, if $\alpha_1=1$, then $x_1x_2m_1$ is divisible by $x_1^2x_2\in A$ because $x_1\in A:(x_1x_2)$, and therefore $x_1x_2m_1$ reduces to $0$. The same holds if $\beta_1=1$. 
 
Now notice that $\alpha_2=\beta_2$. Indeed, if, say, $\alpha_2=0$ and $\beta_2=1$, then $\overline{x_1x_2m_1}=x_1x_2m_1$ (since $x_1x_2m_1$ is squarefree), while $x_1x_2m_2=x_1x_2^2x_3^{\beta_3}\cdots x_{n-1}^{\beta_{n-1}}\to x_1x_3x_4x_3^{\beta_3}\cdots x_{n-1}^{\beta_{n-1}}$.
From this first reduction step it is already clear that $x_2$ does not divide $\overline{x_1x_2m_2}$, while $x_2$ does divide $\overline{x_1x_2m_1}$, which is a contradiction. 
If $\alpha_2=\beta_2=0$, we have $\overline{x_1x_2m_1}=x_1x_2m_1$ and $\overline{x_1x_2m_2}=x_1x_2m_2$, which then implies $m_1=m_2$, which is a contradiction. So we will assume $\boxed{\alpha_2=\beta_2=1}$.
Let $m_1=x_2m_1'$ and $m_2=x_2m_2'$, where $m_1'=x_3^{\alpha_3}\cdots x_{n-1}^{\alpha_{n-1}}$ and $m_2'=x_3^{\beta_3}\cdots x_{n-1}^{\beta_{n-1}}$. Then $x_1x_2m_1=x_1x_2^2m_1'\to x_1x_3x_4m_1'$ and similarly $x_1x_2m_2=x_1x_2^2m_2'\to x_1x_3x_4m_2'$. Clearly, $x_1$ will not be involved in any further reductions. Therefore, it is enough to consider $x_3x_4m_1'$ and $x_3x_4m_2'$. Note that $m_1'$ and $m_2'$ are squarefree monomials in $K[x_3,x_4,\ldots, x_{n-1}]$, and the only reduction rules we can use to reduce them are $x_3^2\to x_4x_5, \ldots, x_{n-3}^2\to x_{n-2}x_{n-1},x_{n-1}^2\to 0, x_{n}\to 0$. This means we are doing a reduction of $x_3x_4m_1'$ and $x_3x_4m_2'$ modulo $A'=(x_3^2-x_4x_5, \ldots, x_{n-3}^2-x_{n-2}x_{n-1},x_{n-1}^2, x_{n})$ inside $K[x_3,x_4,\ldots, x_{n-1}]$. By the induction hypothesis we know that $\overline{x_3x_4m_1'}=\overline{x_3x_4m_2'}\not=0$ implies $m_1'=m_2'$ therefore, $m_1=m_2$, which is a contradiction.



Now it remains to prove the second claim. Let $m=x_1^{\alpha_1}\cdots x_{n-1}^{\alpha_{n-1}}
$ be a squarefree monomial such that $\overline{x_1x_2m}=0$. Assume $m\not\in A+(x_1,x_2x_4\ldots x_{n-3}x_{n-1})$. We first note that $\boxed{\alpha_1=0}$, otherwise we immediately get a contradiction. If $\alpha_2=0$, then $\overline{x_1x_2m_1}=x_1x_2m_1\not=0$. Therefore, $\boxed{\alpha_2=1}$. Let $m=x_2m'$, where $m'=x_3^{\alpha_3}\cdots x_{n-1}^{\alpha_{n-1}}$. Then $x_1x_2m=x_1x_2^2m'\to x_1x_3x_4m'$. Clearly, $x_1$ will not be involved in any further reductions, so it suffices to consider $x_3x_4m'$. Note that $m'$ is a squarefree monomial in $K[x_3,x_4,\ldots, x_{n-1}]$, and the only reduction rules we can use to reduce $x_3x_4m'$ are $x_3^2\to x_4x_5, \ldots, x_{n-3}^2\to x_{n-2}x_{n-1},x_{n-1}^2\to 0, x_{n}\to 0$. In other words, we are doing a reduction of $x_3x_4m'$ modulo $A'=(x_3^2-x_4x_5, \ldots, x_{n-3}^2-x_{n-2}x_{n-1},x_{n-1}^2, x_{n})$ inside $K[x_3,x_4,\ldots, x_{n-1}]$. By the induction hypothesis we know that the squarefree monomials modulo $A'$ which reduce to $0$ are either multiples of $x_3$ or multiples of $x_4x_6\cdots x_{n-1}$. Assume $m'$ is a multiple of $x_3$ modulo $A'$ (therefore, also modulo $A$). Then, since $m=x_2m'$, it follows that $m$ is a multiple of $x_2x_3$, and hence a multiple of $x_1^2$ modulo $A$, which is a contradiction. If $m'$ is a multiple of $x_4\cdots x_{n-1}$, then modulo $A$ we have that 
 $m$ a multiple of $x_2x_4\cdots x_{n-1}$, which is a contradiction.
\end{proof}
\section{The structure of the Rees ring of $I_{\frac{n}{2}}(C_n)$}
The methods and the results in this section will be similar to those in \Cref{sec:n-2}. Note, however, that the roles of $x$-variables and $y$-variables are somewhat swapped compared to \Cref{sec:n-2}: recall that the extra elements $g_k$ of $J$ that occurred in the Gröbner basis computations in \Cref{sec:n-2}  had degree $(1,k)$. In this section, however, the extra elements $g_k$ will have degree $(k,1)$. This is not much of a difference from a Gröbner basis perspective, but we are losing the $x$-condition which allowed us to establish linear quotients before. In fact, $I_{\frac{n}{2}}(C_n)$ does not have a linear resolution except the case $n=4$. The Cohen-Macaulay type and the Hilbert series of $R(I_{\frac{n}{2}}(C_n))$ are also somewhat harder to compute in this case.


For the rest of this section, let $n\ge 4$ be an even integer. Let $L$ be the defining ideal of the symmetric algebra $S(I_{\frac{n}{2}}(C_n))$.
By \cite[Proposition~2.5]{Bd01}  we know that $L=(f_1,\ldots, f_{n-1},g_1)$, where $f_j={x_{\frac{n}{2}+j}y_j}-x_jy_{j+1}$ for $j=1,\ldots,n-1$ and $g_1=x_0y_1-x_{\frac{n}{2}}y_0$. Note that if we extend the definition of $f_j$ to $j=n$, we will obtain $g_1=-f_n$. 
As a preparation for a computation of a Gröbner basis for $J$ (which is the defining ideal of $R(I_{\frac{n}{2}}(C_n))$), we first extend the set of generators of $L$.

\begin{Proposition}
$L=(f_1,\ldots,f_{n-1},g_1,\ldots,g_{\frac{n}{2}-1}),$
where $$g_k={y_k\prod_{i=0}^{k-1}x_i}-y_0\prod_{i=\frac{n}{2}}^{k-1+\frac{n}{2}}x_i \quad \textrm{for}\quad  k=1, \ldots,\frac{n}{2}-1.$$ 
\end{Proposition}
\begin{proof}
We will show by induction that $g_k\in (f_1,\ldots,f_{n-1},g_1)$ for all $k=1,\ldots, \frac{n}{2}-1$. The base case holds by definition. Now, modulo $L$ we have:
\begin{align*}    
y_k\prod_{i=0}^{k-1}x_i&=(x_{k-1}y_k)\prod_{i=0}^{k-2}x_i\overset{f_{k-1}}{=}x_{\frac{n}{2}+k-1}y_{k-1}\prod_{i=0}^{k-2}x_i=x_{\frac{n}{2}+k-1}(x_{k-2}y_{k-1})\prod_{i=0}^{k-3}x_i\\
&\overset{f_{k-2}}{=}(x_{\frac{n}{2}+k-1}x_{\frac{n}{2}+k-2})y_{k-2}\prod_{i=0}^{k-3}x_i=
\cdots\overset{f_1}{=}\prod_{i=1+\frac{n}{2}}^{k-1+\frac{n}{2}}x_i(x_0y_1)\overset{g_1}{=}y_0\prod_{i=\frac{n}{2}}^{k-1+\frac{n}{2}}x_i.
\end{align*}

\end{proof}

Recall that by \Cref{fiber} we have $H=(h_1,\ldots,h_{\frac{n}{2}-1})$, where $$h_l={y_ly_{l+\frac{n}{2}}}-y_0y_{\frac{n}{2}}\quad \textrm{for}\quad l=1, \ldots,\frac{n}{2}-1.$$ We will now describe a Gröbner basis of $L+HT$, which will later turn out to be equal to $J$, the defining ideal of $R(I_{\frac{n}{2}}(C_n))$.

\begin{Theorem}\label{shan}
Let $<$ be the the product order which we defined in \Cref{GB1}
Then the polynomials $$f_1, \ldots, f_{n-1},g_1, \ldots, g_{\frac{n}{2}-1}, h_1, \ldots, h_{\frac{n}{2}-1}$$ form a Gröbner basis for $L+HT$.
\end{Theorem}
The proof of this fact can be found in the Appendix. 

We are now ready to describe the defining ideal $J$ of $R(I_{\frac{n}{2}}(C_n))$. 

\begin{Theorem}
\label{shanfibertype}
 Let as before $L=(f_1,\ldots, f_n)$, where  $f_j=x_{\frac{n}{2}+j}y_j-x_jy_{j+1}$ for $j=1,\ldots, n$. As before, let $H=(h_1,\ldots,h_{\frac{n}{2}-1})$, where $h_l={y_ly_{l+\frac{n}{2}}}-y_0y_{\frac{n}{2}}\quad \textrm{for}\quad l=1, \ldots,\frac{n}{2}-1.$
Then $J=L+HT$.
\end{Theorem}
\begin{proof}
Let $J'=(f_0, \ldots,f_{n-1}, h_1, \ldots, h_{\frac{n}{2}-1})$. We apply the same strategy as in the proof of \Cref{Rees} in order to show that  $J'=J$. It follows from the description of the Gr\"{o}bner basis of $J'$,  given in \Cref{shan}, that $x_{\frac{n}{2}}$ is a nonzero-divisor modulo $\ini(J')$. Then $x_{\frac{n}{2}}$ is a  nonzero-divisor  modulo $J'$,  as well. By symmetry it follows that all $x_i$ are nonzero-divisors modulo $J'$. Consequently,  $x=\prod_{i=1}^nx_i$ is a nonzero-divisor modulo $J'$. Therefore, the natural homomorphism $T/J'\To (T/J')_x$ is injective. Thus $J'$ will be a prime ideal, if $J'T_x$ is a prime ideal. 
The relations $f_1,\ldots, f_{n-1}$ show that $y_i=w_iy_{i+1}$ modulo $J'T_x$ for $i=0,\ldots, n-1$, where $w_i$ are some monomials in $x_i^{\pm 1}$. 
Thus, for each $1\le i\le n-1$ we get $y_i=w_iy_{i+1}=w_iw_{i+1}y_{i+2}$. Proceeding in this way,  we see that for all $i=1,\ldots, n-1$ we have $y_i=v_iy_0$ 
for some unit $v_i$,  where each $v_i$ is a product of certain $x_k^{\pm 1}$ and $v_n=1$.  

It follows that
\begin{multline*}
T_x/J'T_x\iso K[y_n, x_1^{\pm 1}, \ldots, x_n^{\pm 1}]/(x_{\frac{n}{2}}y_n-x_nv_1y_n, \\
\{ v_lv_{\frac{n}{2}+l}y_n^2-v_nv_{\frac{n}{2}}y_n^2: l=1,\ldots, \frac{n}{2}-1\}).
\end{multline*}



Since $J'\subseteq J=\ker(\phi)$ sending $y_n\mapsto x_nx_1\cdots x_{\frac{n}{2}-1}s$, we conclude that these binomials vanish under this substitution.
Since they are homogeneous in $y_n$, it follows that the coefficients  of $y_n$ in the first binomial must be equal, and the coefficients of $y_n^2$ in the remaining binomials must also be equal.

This implies that $T_x/J'T_x\iso K[y_n, x_1^{\pm 1}, \ldots, x_n^{\pm 1}]$, from which we conclude that $J'$ is a prime ideal of height $n+1$. Since $J'\subseteq J$ and since $\height J=n+1$, we have $J'=J$, as desired. 
\end{proof}
\begin{Corollary}
\label{normalCM} 
 Let $n\geq 4$ be an even integer and let $I=I_{\frac{n}{2}}(C_n)$. Then the following holds:
 \begin{enumerate}
\item[{\em (a)}] $R(I)$ is a normal Cohen--Macaulay domain. 
\item[{\em (b)}] $I$  is of fiber type.  
 \end{enumerate}
 \end{Corollary}

\begin{proof}
(a)  Let $J$ as before be the defining ideal of $R(I)$.  By \Cref{shan}, $J$ has a squarefree initial ideal. By \cite[Proposition~13.15]{St} this implies that $R(I)$ is a normal. Then, since $R(I)$ is a toric ring, we obtain  that $R(I)$ is a  Cohen--Macaulay domain,  see \cite[Theorem~1]{Ho}. 

(b) This can be seen from the generators of $J$ given in \Cref{shanfibertype}.

\end{proof}

\section{Summary of results and open problems}

In this section we discuss some observations based on computations in \textit{Macaulay2}. We get the following table: the entry in the $t$th column and the $n$th row (where rows start at $n=3$) is $L$, $F$ or $\times$, which means that $I_t(C_{n})$ is of linear type, fiber type or neither, respectively.
\begin{figure}[H]
\begin{tikzpicture}[square/.style={regular polygon,regular polygon sides=4}]
\node [thin] at (9.6,0) {$13$};
\node [thin] at (8.8,0) {$12$};
\node [thin] at (8,0) {$11$};
\node [thin] at (7.2,0) {$10$};
\node [thin] at (6.4,0) {$9$};
\node [thin] at (5.6,0) {$8$};
\node [thin] at (4.8,0) {$7$};
\node [thin] at (4,0) {$6$};
\node [thin] at (3.2,0) {$5$};
\node [thin] at (2.4,0) {$4$};
\node [thin] at (1.6,0) {$3$};
\node [thin] at (0.8,0) {$2$};
\node [thin] at (0,0) {$1$};

\draw[thin] (-1.3,-0.5)--(10.1,-0.5);
\node [thin] at (-0.5,0.2) {$t$};
\draw[thin] (-0.3,-0.5)--(-0.3-1*cos 45,-0.5+1*sin 45);
\node [thin] at (-1,-0.2) {$n$};
\draw[thin] (-0.3,0.5)--(-0.3,-9.9);
\filldraw[color=yellow, fill=yellow, very thick] (-0.15,-0.6)rectangle (0.15,-9.8);\node[thin] at (0,-10.2){$\textcolor{yellow}{t=1}$};
\draw[yellow, ultra thick] (0.68,-0.53)--(9.77,-9.62); 
\draw[yellow, ultra thick] (0.65,-0.56)--(9.74,-9.65);
\draw[yellow, ultra thick] (0.62,-0.59)--(9.71,-9.68); 
\draw[yellow, ultra thick] (0.59,-0.62)--(9.68,-9.71);
\draw[yellow, ultra thick] (0.56,-0.65)--(9.65,-9.74); 
\draw[yellow, ultra thick] (0.53,-0.68)--(9.62,-9.77);
\draw[yellow, ultra thick] (0.5,-0.71)--(9.59,-9.8); 
\node[thin] at (10.5,-10.2){$\textcolor{yellow}{t=n-1}$};

\filldraw[color=pink, fill=pink, very thick] (0.65,-0.6)rectangle (0.95,-9.8);\node[thin] at (1,-10.2){$\textcolor{pink}{t=2}$};
\draw[pink, ultra thick] (0.03,-0.62)--(9,-9.59); 
\draw[pink, ultra thick] (0,-0.65)--(8.97,-9.62);
\draw[pink, ultra thick] (-0.03,-0.68)--(8.94,-9.65); 
\draw[pink, ultra thick] (-0.06,-0.71)--(8.91,-9.68);
\draw[pink, ultra thick] (-0.09,-0.74)--(8.88,-9.71); 
\draw[pink, ultra thick] (-0.12,-0.77)--(8.85,-9.74);
\draw[pink, ultra thick] (-0.15,-0.8)--(8.82,-9.77); 
\draw[pink, ultra thick] (-0.18,-0.83)--(8.79,-9.8); 
\node[thin] at (8.8,-10.2){$\textcolor{pink}{t=n-2}$};

\filldraw[color=cyan, fill=cyan, very thick] (0.65,-1.4)rectangle (0.95,-2.6);
\filldraw[color=cyan, fill=cyan, very thick] (1.45,-3)rectangle (1.75,-4.2);
\filldraw[color=cyan, fill=cyan, very thick] (2.25,-4.6)rectangle (2.55,-5.8);
\filldraw[color=cyan, fill=cyan, very thick] (3.05,-6.2)rectangle (3.35,-7.4);
\filldraw[color=cyan, fill=cyan, very thick] (3.85,-7.8)rectangle (4.15,-9.0);
\filldraw[color=cyan, fill=cyan, very thick] (4.65,-9.4)rectangle (4.95,-9.8);
\node[thin] at (4.85,-10.2){$\textcolor{cyan}{t=\lfloor\frac{n}{2}\rfloor}$};

\node [thin] at (-0.8,-0.8) {$C_3$};
\node [thin] at (-0.8,-1.6) {$C_4$};
\node [thin] at (-0.8,-2.4) {$C_5$};
\node [thin] at (-0.8,-3.2) {$C_6$};
\node [thin] at (-0.8,-4) {$C_7$};
\node [thin] at (-0.8,-4.8) {$C_8$};
\node [thin] at (-0.8,-5.6) {$C_9$};
\node [thin] at (-0.8,-6.4) {$C_{10}$};
\node [thin] at (-0.8,-7.2) {$C_{11}$};
\node [thin] at (-0.8,-8) {$C_{12}$};
\node [thin] at (-0.8,-8.8) {$C_{13}$};
\node [thin] at (-0.8,-9.6) {$C_{14}$};
\node [thin] at (0,-0.8) {$L$};
\node [thin] at (0.8,-0.8) {$L$};
\node [thin] at (0,-1.6) {$L$};
\node [thin] at (0.8,-1.6) {$F$};
\node [thin] at (1.6,-1.6) {$L$};
\node [thin] at (0,-2.4) {$L$};
\node [thin] at (0.8,-2.4) {$L$};
\node [thin] at (1.6,-2.4) {$L$};
\node [thin] at (2.4,-2.4) {$L$};
\node [thin] at (0,-3.2) {$L$};
\node [thin] at (0.8,-3.2) {$F$};
\node [thin] at (1.6,-3.2) {$F$};
\node [thin] at (2.4,-3.2) {$F$};
\node [thin] at (3.2,-3.2) {$L$};
\node [thin] at (0,-4) {$L$};
\node [thin] at (0.8,-4) {$L$};
\node [thin] at (1.6,-4) {$L$};
\node [thin] at (2.4,-4) {$\times$};
\draw (2.15,-3.75) rectangle (2.65,-4.25);
\node [thin] at (3.2,-4) {$L$};
\node [thin] at (4,-4) {$L$};
\node [thin] at (0,-4.8) {$L$};
\node [thin] at (0.8,-4.8) {$F$};
\node [shape=circle, draw, inner sep=1.5pt] at (1.6,-4.8) {$\times$};
\node [thin] at (2.4,-4.8) {$F$};
\node [thin] at (3.2,-4.8) {$\times$};
\draw (2.95,-4.55) rectangle (3.45,-5.05);
\node [thin] at (4,-4.8) {$F$};
\node [thin] at (4.8,-4.8) {$L$};
\node [thin] at (0,-5.6) {$L$};
\node [thin] at (0.8,-5.6) {$L$};
\node [shape=circle, draw, dashed, inner sep=1.5pt] at (1.6,-5.6) {$F$};
\node [thin] at (2.4,-5.6) {$L$};
\node [thin] at (3.2,-5.6) {$\times$};
\draw (2.95,-5.35) rectangle (3.45,-5.85);
\node [thin] at (4,-5.6) {$F$};
\draw [dashed](3.75,-5.35) rectangle (4.25,-5.85);
\node [thin] at (4.8,-5.6) {$L$};
\node [thin] at (5.6,-5.6) {$L$};
\node [thin] at (0,-6.4) {$L$};
\node [thin] at (0.8,-6.4) {$F$};
\node [thin] at (1.6,-6.4) {$L$};
\node [thin] at (2.4,-6.4) {$\times$};
\node [thin] at (3.2,-6.4) {$F$};
\node [thin] at (4,-6.4) {$\times$};
\draw [dashed](3.75,-6.15) rectangle (4.25,-6.65);
\node [thin] at (4.8,-6.4) {$\times$};
\draw (4.55,-6.15) rectangle (5.05,-6.65);
\node [thin] at (5.6,-6.4) {$F$};
\node [thin] at (6.4,-6.4) {$L$};
\node [thin] at (0,-7.2) {$L$};
\node [thin] at (0.8,-7.2) {$L$};
\node [shape=circle, draw, inner sep=1.5pt] at (1.6,-7.2) {$\times$};
\node [shape=circle, draw, inner sep=1.5pt] at (2.4,-7.2) {$\times$};
\node [thin] at (3.2,-7.2) {$L$};
\node [thin] at (4,-7.2) {$\times$};
\draw (3.75,-6.95) rectangle (4.25,-7.45);
\node [thin] at (4.8,-7.2) {$\times$};
\draw (4.55,-6.95) rectangle (5.05,-7.45);
\node [thin] at (5.6,-7.2) {$\times$};
\draw (5.35,-6.95) rectangle (5.85,-7.45);
\node [thin] at (6.4,-7.2) {$L$};
\node [thin] at (7.2,-7.2) {$L$};
\node [thin] at (0,-8) {$L$};
\node [thin] at (0.8,-8) {$F$};
\node [shape=circle, draw, dashed, inner sep=1.5pt] at (1.6,-8) {$F$};
\node [shape=circle, draw, dashed, inner sep=1.5pt] at (2.4,-8) {$F$};
\node [shape=circle, draw, inner sep=1.5pt] at (3.2,-8) {$\times$};
\node [thin] at (4,-8) {$F$};
\node [thin] at (4.8,-8) {$\times$};
\draw (4.55,-7.75) rectangle (5.05,-8.25);
\node [thin] at (5.6,-8) {$\times$};
\draw [dashed](5.35,-7.75) rectangle (5.85,-8.25);
\node [thin] at (6.4,-8) {$\times$};
\draw [dashed](6.15,-7.75) rectangle (6.65,-8.25);
\node [thin] at (7.2,-8) {$F$};
\node [thin] at (8,-8) {$L$};
\node [thin] at (0,-8.8) {$L$};
\node [thin] at (0.8,-8.8) {$L$};
\node [thin] at (1.6,-8.8) {$L$};
\node [thin] at (2.4,-8.8) {$L$};
\node [thin] at (3.2,-8.8) {$\times$};
\draw (3.2,-8.8) ellipse (0.4cm and 0.2cm);
\node [thin] at (4,-8.8) {$L$};
\node [thin] at (4.8,-8.8) {$\times$};
\draw (4.55,-8.55) rectangle (5.05,-9.05);
\node [thin] at (5.6,-8.8) {$\times$};
\draw (5.35,-8.55) rectangle (5.85,-9.05);
\node [thin] at (6.4,-8.8) {$\times$};
\draw (6.15,-8.55) rectangle (6.65,-9.05);
\node [thin] at (7.2,-8.8) {$\times$};
\draw (6.95,-8.55) rectangle (7.45,-9.05);
\node [thin] at (8,-8.8) {$L$};
\node [thin] at (8.8,-8.8) {$L$};
\node [thin] at (0,-9.6) {$L$};
\node [thin] at (0.8,-9.6) {$F$};
\node [shape=circle, draw, inner sep=1.5pt] at (1.6,-9.6) {$\times$};
\node [thin] at (2.4,-9.6) {$\times$};
\node [shape=circle, draw, inner sep=1.5pt] at (3.2,-9.6) {$\times$};
\node [thin] at (4,-9.6) {$\times$};
\node [thin] at (4.8,-9.6) {$F$};
\node [thin] at (5.6,-9.6) {$\times$};
\draw [dashed](5.35,-9.35) rectangle (5.85,-9.85);
\node [thin] at (6.4,-9.6) {$\times$};
\draw (6.15,-9.35) rectangle (6.65,-9.85);
\node [thin] at (7.2,-9.6) {$\times$};
\draw [dashed](6.95,-9.35) rectangle (7.45,-9.85);
\node [thin] at (8,-9.6) {$\times$};
\draw (7.75,-9.35) rectangle (8.25,-9.85);
\node [thin] at (8.8,-9.6) {$F$};
\node [thin] at (9.6,-9.6) {$L$};
\end{tikzpicture}
\caption{When $I_t(C_n)$ is of linear type or fiber type or neither?}
\label{figure}
\end{figure}

The following positive results are either known or proven in this paper:
\begin{itemize}
    \item $I_1(C_n)$ is of linear type for all $n$ since it is generated by a regular sequence;
    \item $I_{n-1}(C_n)$ is of linear type for all $n$ by \cite[Theorem~3.2]{Bd01}.

\noindent In \Cref{figure} above, the cases $t=1$ and $t=n-1$  are marked with yellow lines.
    \item $I_2(C_n)$ is of linear type if $n$ is odd by \cite[Theorem~3.4]{Vil90}, and of fiber type (but not linear type) if $n$ is even by \cite[Theorem~8.2.1]{Vil01};
    \item $I_{n-2}(C_n)$ is of linear type if $n$ is odd by \cite[Theorem~3.4]{Bd01}, and of fiber type (but not linear type) if $n$ is even by \Cref{niceinitial}.
    
    \noindent In \Cref{figure} above, the cases $t=2$ and $t=n-2$  are marked with pink lines.
    \item $I_{\lfloor\frac{n}{2}\rfloor}(C_n)$ is of linear type if $n$ is odd by \cite[Corollary~3.6]{Bd01} and of fiber type (but not linear type) if $n$ is even by \Cref{normalCM}.
    
    \noindent In \Cref{figure} above, the case $t=\lfloor\frac{n}{2}\rfloor$  is marked with blue vertical line segments.

\end{itemize}
As mentioned in the introduction, there are some negative results on this matter proved in \cite{Bd01}:
\begin{itemize}
    \item If $gcd(n,t)=1$
      and $[\frac{n-1}{2}]<t<n-2$, then $I_t(C_n)$ is not of linear type.

      \noindent Since $gcd(n,t)=1$, we conclude that in this case we have neither linear, no fiber type. In \Cref{figure}, pairs $(n,t)$ which satisfy this condition are marked with squares.
      \item If $gcd(n,t)=1$, $tl\equiv 1 \pmod n$ and $1<t,l\le[\frac{n-1}{2}]$, then $I_t(C_n)$ is not of linear type.
      
      \noindent Again, we conclude that in this case we have neither linear, no fiber type. In \Cref{figure}, such pairs $(n,t)$ are marked with circles.
      \item If $gcd(n,t)=1$, $tl\equiv 1 \pmod n$, $1<t\le[\frac{n-1}{2}]$, $[\frac{n-1}{2}]<l<n$ and $n=p(n-l)+a$, $2\leq a<n-l$, then $I_t(C_n)$ is not of linear type.

      \noindent The only relevant case in \Cref{figure} is $(n,t)=(13,5)$ and it is marked with an ellipse.
\end{itemize}

The table above leads us to the following:
\begin{Conjecture}
     \emph{(1)} Suppose  $2< t<\lfloor\frac{n}{2}\rfloor$. Then $I_t(C_n)$ is of fiber type, but not of linear type, if and only if $t \mid n$.

     \emph{(2)} Suppose $gcd(n,t)>1$ and $\lfloor \frac{n}{2}\rfloor<t<n-2$. Then $I_t(C_n)$ is not of fiber type (and hence not of linear type) except when $n=9$ and $t=6$.
\end{Conjecture}
In \Cref{figure}, the pairs $(n,t)$ covered by first part of this conjecture are marked with dashed circles, and those covered by the second part are marked with dashed squares. Finally, cases that are not covered by anything of the above are unmarked. 


\section*{Acknowledgement}
This project was started when the third author was visiting the first and the second author at the University of Duisburg-Essen. The third author thanks the local research group for their hospitality, and would like to acknowledge International Exchange and Overseas Study Scholarship of Soochow University which allowed her to perform this trip.

The first author was supported by a fellowship from the Wenner-Gren Foundations (grant WGF2022-0052).

\vspace{3mm}
\noindent \textbf{Data availability:} The data used to support the findings of this study are included within the article.

\vspace{3mm}
\noindent \textbf{Statement:} On behalf of all authors, the corresponding author states that there is no conflict of interest.

\section*{Appendix}
\emph{\Cref{GB1}.} Recall that $L=(f_1,\ldots,f_{n-1},g_1,\ldots,g_{\lfloor\frac{n}{2}\rfloor}),$
where 
$$f_j=\underline{x_{j-2}y_j}-x_jy_{j+1}, \quad g_k=\underline{{x_{2k-2}\prod_{\substack{i\in [0,2k)\\i \text{ odd}}}y_i}}-x_{n-2}\prod_{\substack{i\in [0,2k)\\i \text{ even}}}y_i.$$

These generators form a Gröbner basis of $L$ with respect to any monomial order $<$ on $T$ which is the product order of $<_1$ and $<_2$, where $<_1$ is the reverse lexicographic order on $R$ induced by $y_1>y_2>\cdots > y_{n-1}>y_0$, and $<_2$ is an arbitrary monomial order on $S$.

\begin{proof}[Proof of \Cref{GB1}]

The proof will be performed in three steps.

Step 1: \emph{The $S-$pairs $S(g_{k_1},g_{k_2})$ all reduce to $0$.}
Let $1\le k-s<k\le \lfloor\frac{n}{2}\rfloor$. We would like to show that $S(g_k, g_{k-s})$ reduces to $0$. 
We have: 
\begin{align*}
&S(g_k,g_{k-s})\\&=S\left(\underline{x_{2k-2}\prod_{\substack{i\in [0,2k)\\i \text{ odd}}}y_i}-x_{n-2}\prod_{\substack{i\in [0,2k)\\i \text{ even}}}y_i,\quad \underline{x_{2k-2s-2}\prod_{\substack{i\in [0,2k-2s)\\i \text{ odd}}}y_i}-x_{n-2}\prod_{\substack{i\in [0,2k-2s)\\i \text{ even}}}y_i\right)\\
&=g_kx_{2k-2s-2}-g_{k-s}x_{2k-2}\prod_{\substack{i\in [2k-2s,2k)\\i \text{ odd}}}y_i\\
&=x_{n-2}x_{2k-2}\prod_{\substack{i\in [2k-2s,2k)\\i \text{ odd}}}y_i \prod_{\substack{i\in [0,2k-2s)\\i \text{ even}}}y_i-x_{n-2}x_{2k-2s-2}\prod_{\substack{i\in [0,2k)\\i \text{ even}}}y_i\\
&=x_{n-2}\prod_{\substack{i\in [0,2k-2s)\\i \text{ even}}}y_i\left(x_{2k-2}\prod_{\substack{i\in [2k-2s,2k)\\i \text{ odd}}}y_i-x_{2k-2s-2}\prod_{\substack{i\in [2k-2s,2k)\\i \text{ even}}}y_i\right).
\end{align*}
We will prove by induction on $s$ that the binomial in parentheses reduces to $0$. Since $s\ge 1$, the interval $[2k-2s,2k)$ is nonempty and we take out $y_{2k-2s+1}$ and $y_{2k-2s}$ and rewrite the binomial in parentheses as follows:
\begin{align*}
x_{2k-2}y_{2k-2s+1}\prod_{\substack{i\in [2k-2s+2,2k)\\i \text{ odd}}}y_i-\underline{x_{2k-2s-2}y_{2k-2s}\prod_{\substack{i\in [2k-2s+2,2k)\\i \text{ even}}}y_i}.    
\end{align*}
Since $k-s\ge 1$, $2k\le n$ and $s\ge 1$, we have $2\le 2k-2s\le n-2$, and thus $2k-2s$ lies inside the range where $f_i$'s are defined. Therefore, one can say that $x_{2k-2s-2}y_{2k-2s}$ is the leading term of $f_{2k-2s}=\underline{x_{2k-2s-2}y_{2k-2s}}-x_{2k-2s}y_{2k-2s+1}$. Therefore, reducing the leading term of our binomial using $f_{2k-2s}$, we obtain:
\begin{align*}
&x_{2k-2}y_{2k-2s+1}\prod_{\substack{i\in [2k-2s+2,2k)\\i \text{ odd}}}y_i-x_{2k-2s}y_{2k-2s+1}\prod_{\substack{i\in [2k-2s+2,2k)\\i \text{ even}}}y_i\\=&y_{2k-2s+1}\left(x_{2k-2}\prod_{\substack{i\in [2k-2(s-1),2k)\\i \text{ odd}}}y_i-x_{2k-2(s-1)-2}\prod_{\substack{i\in [2k-2(s-1),2k)\\i \text{ even}}}y_i\right),    
\end{align*}
and the binomial in parentheses is exactly the one we wanted to reduce, with $s\to s-1$. Note that if $s=1$ it equals $0$, which settles the base case.

Step 2: \emph{The $S-$pairs $S(f_j,g_k)$ all reduce to $0$.} Recall that
$$f_j=\underline{x_{j-2}y_j}-x_jy_{j+1}, \quad g_k=\underline{x_{2k-2}\prod_{\substack{i\in [0,2k)\\i \text{ odd}}}y_i}-x_{n-2}\prod_{\substack{i\in [0,2k)\\i \text{ even}}}y_i.$$ We need to consider only the cases when the leading terms of these binomials are not coprime. This can happen in the following two cases: 

Case 1: The two leading terms are not coprime because of a common $x-$factor, in other words, if $j=2k$. Note that this case is impossible if $2k=n$ since $j$ only takes values until $n-1$.
Then we have 
\begin{align*}
S(f_{2k},g_{k})&=S\left(\underline{x_{2k-2}y_{2k}}-x_{2k}y_{2k+1},\quad\underline{x_{2k-2}\prod_{\substack{i\in [0,2k)\\i \text{ odd}}}y_i}-x_{n-2}\prod_{\substack{i\in [0,2k)\\i \text{ even}}}y_i\right)\\
&=f_{2k}\prod_{\substack{i\in [0,2k)\\i \text{ odd}}}y_i-g_ky_{2k}=x_{n-2}y_{2k}\prod_{\substack{i\in [0,2k)\\i \text{ even}}}y_i-x_{2k}y_{2k+1}\prod_{\substack{i\in [0,2k)\\i \text{ odd}}}y_i.
\end{align*}
We recall that the indices of variables should be interpreted modulo $n$, while the intervals should be interpreted in the usual way. For example, if $n=2k+1$, then $y_{2k+1}$ should be interpreted as $y_0$ and therefore in this case we avoid writing for instance $y_{2k+1}\prod_{\substack{i\in [0,2k)\\i \text{ odd}}}y_i=\prod_{\substack{i\in [0,2k+2)\\i \text{ odd}}}y_i$ in order to avoid the confusion of $y_0$ being present in the product of $y_i$ with odd indices. That is why all the intervals we use in products have to be subsets of $[0,n)$. 

In order to further reduce this binomial, we need to consider two final subcases:

Subcase 1.a: $2k=n-1$. In this case our binomial becomes
$$x_{n-2}y_{n-1}\prod_{\substack{i\in [0,n-1)\\i \text{ even}}}y_i-\underline{x_{n-1}y_{0}\prod_{\substack{i\in [0,n-1)\\i \text{ odd}}}y_i}.$$
We will show that the leading term reduces to the other term:
\begin{align*}
x_{n-1}y_{0}\prod_{\substack{i\in [0,n-1)\\i \text{ odd}}}y_i
&=\prod_{\substack{i\in [0,2)\\i \text{ even}}}y_i\cdot (x_{n-1}y_1)\cdot\prod_{\substack{i\in [2,n-1)\\i \text{ odd}}}y_i\\
&\xrightarrow{f_1} \prod_{\substack{i\in [0,2)\\i \text{ even}}}y_i\cdot x_1y_2\cdot\prod_{\substack{i\in [2,n-1)\\i \text{ odd}}}y_i=\prod_{\substack{i\in [0,4)\\i \text{ even}}}y_i\cdot (x_1y_3)\cdot\prod_{\substack{i\in [4,n-1)\\i \text{ odd}}}y_i\\
&\xrightarrow{f_3} \prod_{\substack{i\in [0,4)\\i \text{ even}}}y_i\cdot x_3y_4\cdot\prod_{\substack{i\in [4,n-1)\\i \text{ odd}}}y_i=\prod_{\substack{i\in [0,6)\\i \text{ even}}}y_i\cdot (x_3y_5)\cdot\prod_{\substack{i\in [6,n-1)\\i \text{ odd}}}y_i\\
&\xrightarrow{f_5}\cdots\\
&\xrightarrow{f_{n-2}}\prod_{\substack{i\in [0,n-1)\\i \text{ even}}}y_i\cdot x_{n-2}y_{n-1} =x_{n-2}y_{n-1}\prod_{\substack{i\in [0,n-1)\\i \text{ even}}}y_i.
\end{align*}
In order to see the last reduction step, note that right after the reduction by $f_{2s+1}$ the monomial is $$\prod_{\substack{i\in [0,2s+2)\\i \text{ even}}}y_i\cdot (x_{2s+1}y_{2s+2})\cdot\prod_{\substack{i\in [2s+2,n-1)\\i \text{ odd}}}y_i.$$
We set $s=k-1$ and obtain that after the reduction by $f_{2s+1}=f_{2k-1}=f_{n-2}$ we get
$$\prod_{\substack{i\in [0,n-1)\\i \text{ even}}}y_i\cdot (x_{n-2}y_{n-1})\cdot\prod_{\substack{i\in [n-1,n-1)\\i \text{ odd}}}y_i,$$
which is precisely what we need.

Subcase 1.b: $2k\le n-2$. In this case we have $2k+2\le n$ and so we can use the interval $[0,2k+2)$ to rewrite  our binomial as $$x_{n-2}\prod_{\substack{i\in [0,2k+2)\\i \text{ even}}}y_i-x_{2k}\prod_{\substack{i\in [0,2k+2)\\i \text{ odd}}}y_i.$$
Note that this is simply $-g_{k+1}$. Also note that the indices of $g$'s are only allowed to up until $\lfloor\frac{n}{2}\rfloor$, and so index $k+1$ is allowed since $2(k+1)=2k+2\le n$. This is not the case in the Subcase 1.a, where $g_{k+1}$ merely does not exist.

Case 2: The two leading terms are not coprime because of a common $y-$factor, in other words, if $j$ is an odd number belonging to $[0,2k)$. Then we get 

\begin{align*}
S(f_j,g_k)&=S\left(\underline{x_{j-2}y_j}-x_jy_{j+1}, \quad \underline{x_{2k-2}\prod_{\substack{i\in [0,2k)\\i \text{ odd}}}y_i}-x_{n-2}\prod_{\substack{i\in [0,2k)\\i \text{ even}}}y_i\right)\\
&=f_jx_{2k-2}\prod_{\substack{i\in [0,2k)\setminus\{j\}\\i \text{ odd}}}y_i-g_kx_{j-2}\\
&=x_{j-2}x_{n-2}\prod_{\substack{i\in [0,2k)\\i \text{ even}}}y_i-\underline{x_{2k-2}x_jy_{j+1}\prod_{\substack{i\in [0,2k)\setminus\{j\}\\i \text{ odd}}}y_i}
\end{align*}

Note that the underlined term is indeed the leading term: unless $j+1=n$, we see that $y_0$  is present only in the support of the first term (making it the smaller one). If $j+1=n$, and noting that $j=n-1$ is an odd number belonging to $[0,2k)$, we conclude that $2k=n$. Then the above binomial becomes $x_{n-3}x_{n-2}\prod_{\substack{i\in [0,n)\\i \text{ even}}}y_i-{x_{n-2}x_{n-1}y_0\prod_{\substack{i\in [0,n-2)\\i \text{ odd}}}y_i}$. We see that $y_0$ is present on both sides, then $y_{n-1}$ is not present anywhere, and $y_{n-2}$ is only present in the first term (making it the smaller one). We will show that the leading term reduces to the  other one:
\begin{align*}
  &x_{2k-2}x_jy_{j+1}\prod_{\substack{i\in [0,2k)\setminus\{j\}\\i \text{ odd}}}y_i
  =x_{2k-2}\prod_{\substack{i\in [j+1,j+3)\\i \text{ even}}}y_i\cdot x_jy_{j+2}\cdot\prod_{\substack{i\in [0,2k)\setminus\{j,j+2\}\\i \text{ odd}}}y_i\\
  &\xrightarrow{f_{j+2}} x_{2k-2}\prod_{\substack{i\in [j+1,j+3)\\i \text{ even}}}y_i\cdot x_{j+2}y_{j+3}\cdot\prod_{\substack{i\in [0,2k)\setminus\{j,j+2\}\\i \text{ odd}}}y_i\\
  &=x_{2k-2}\prod_{\substack{i\in [j+1,j+5)\\i \text{ even}}}y_i\cdot x_{j+2}y_{j+4}\cdot\prod_{\substack{i\in [0,2k)\setminus\{j,j+2,j+4\}\\i \text{ odd}}}y_i\\
  &\xrightarrow{f_{j+4}} x_{2k-2}\prod_{\substack{i\in [j+1,j+5)\\i \text{ even}}}y_i\cdot x_{j+4}y_{j+5}\cdot\prod_{\substack{i\in [0,2k)\setminus\{j,j+2,j+4\}\\i \text{ odd}}}y_i\\
  &=x_{2k-2}\prod_{\substack{i\in [j+1,j+7)\\i \text{ even}}}y_i\cdot x_{j+4}y_{j+6}\cdot\prod_{\substack{i\in [0,2k)\setminus\{j,j+2,j+4,j+6\}\\i \text{ odd}}}y_i
  \xrightarrow{\cdots}
\end{align*}
We see that right after the reduction by $j+2s$ the monomial is 
$$x_{2k-2}\prod_{\substack{i\in [j+1,j+2s+1)\\i \text{ even}}}y_i\cdot x_{j+2s}y_{j+2s+1}\cdot\prod_{\substack{i\in [0,2k)\setminus\{j,j+2,\ldots,j+2s\}\\i \text{ odd}}}y_i.$$
We do it as many times as it is needed to remove $2k-1$ from the set of odd numbers. In other words, if we set $s=\frac{2k-1-j}{2}$, then  $j+2s=2k-1$ (which is in the range of the allowed indices for $f$'s) and we get
$$x_{2k-2}\prod_{\substack{i\in [j+1,2k)\\i \text{ even}}}y_i\cdot x_{2k-1}y_{2k}\cdot\prod_{\substack{i\in [0,j-1)\\i \text{ odd}}}y_i.$$

It is not difficult to notice that any monomial of type $x_{i}x_{i+1}y_{i+2}$ (for any $i=0,1,\ldots, n-1$, and indices interpreted modulo $n$) reduces to $x_{n-2}x_{n-1}y_0$. Indeed, 
$$
x_{n-1}x_{0}y_1\xrightarrow{f_1}x_0x_1y_2\xrightarrow{f_2}x_1x_2y_3\xrightarrow{\ldots}\ldots\xrightarrow{f_{n-2}}x_{n-3}x_{n-2}y_{n-1}\xrightarrow{f_{n-1}}x_{n-2}x_{n-1}y_{0}.
$$
Therefore, in the monomial above we may replace $x_{2k-2}x_{2k-1}y_{2k}$ with $x_{n-2}x_{n-1}y_{0}$ and get
\begin{align*}
    x_{n-2}\prod_{\substack{i\in [j+1,2k)\\i \text{ even}}}y_i\cdot\left[ x_{n-1}y_{0}\cdot\prod_{\substack{i\in [0,j-1)\\i \text{ odd}}}y_i\right]
\end{align*}
Note that the monomial in square brackets is almost exactly the same as the leading term in Subcase~1.a, with the only difference that here the product is taken over a shorter interval. The reduction process of this monomial will be exactly the same as before, but with fewer steps. Analogously to Subcase~1.a, we do reductions with $f_1,f_3,\ldots$, and we
 see that right after the reduction by $f_{2s+1}$ the monomial in square brackets is $$\prod_{\substack{i\in [0,2s+2)\\i \text{ even}}}y_i\cdot (x_{2s+1}y_{2s+2})\cdot\prod_{\substack{i\in [2s+2,j-1)\\i \text{ odd}}}y_i.$$
if $j=1$, this monomial already equals the one we want to reduce to and we are done. Otherwise  we set $s=\frac{j-3}{2}$ and obtain that after the reduction by $f_{2s+1}=f_{j-2}$ we get

$$\prod_{\substack{i\in [0,j-1)\\i \text{ even}}}y_i\cdot (x_{j-2}y_{j-1})\cdot\prod_{\substack{i\in [j-1,j-1)\\i \text{ odd}}}y_i=\prod_{\substack{i\in [0,j+1)\\i \text{ even}}}y_i\cdot x_{j-2}.$$
Now that we have reduced the part in square brackets, the entire monomial then reduces to 
$$
x_{n-2}\prod_{\substack{i\in [j+1,2k)\\i \text{ even}}}y_i\prod_{\substack{i\in [0,j+1)\\i \text{ even}}}y_i\cdot x_{j-2}=x_{j-2}x_{n-2}\prod_{\substack{i\in [0,2k)\\i \text{ even}}}y_i,
$$
which is precisely what we need.

Step 3: \emph{The $S-$pairs $S(f_i,f_j)$ all reduce to $0$.}
This is obvious since the leading terms of $f_i$ and $f_j$ are coprime for $i\not=j$.

\end{proof}

\emph{\Cref{GBfibers}.}
Let $n$ be even. 
Then the polynomials $\{f_1,\ldots,f_{n-1},g_1,\ldots,g_{\frac{n}{2}},h\}$ form a Gröbner basis for $L+HT$ with respect to any monomial order given in \Cref{GB1}. Here, as before, $H=(h)$, where $$h=\underline{\prod_{\substack{i\in [0,n)\\i \text{ odd}}}y_i}-\prod_{\substack{i\in [0,n)\\i \text{ even}}}y_i.$$

\begin{proof}[Proof of \Cref{GBfibers}]
Let $n=2s$. We only need to check that the $S$-pairs involving $h$ reduce to $0$. Note that $g_s=x_{n-2}h$, and so the $S$-pairs involving $h$ will reduce very similarly to those involving $g_s$. We will briefly go through all the cases, essentially mimicking the proof of \Cref{GB1}.

 Step 1: \emph{The $S-$pairs $S(h, g_k)$ all reduce to $0$.}
 If $k=s$, there is nothing to check, so we will assume $k=s-a$, $a\ge 1$. It is not hard to verify that $S(g_s,g_{s-a})=x_{n-2}S(h,g_{s-a})$. 
 


In Step~1 in the proof of \Cref{GB1} we computed all $S(g_k,g_{k-a})$, the result was divisible by $x_{n-2}$, which we factored out, and the remaining binomial reduced to $0$. Therefore, with the exact same reduction process $S(h,g_{s-a})$ will reduce to $0$.

Step 2: \emph{The $S-$pairs $S(f_j, h)$ all reduce to $0$.}
We will mimic the proof of Step~2 of \Cref{GB1}.
Recall that
$$f_j=\underline{x_{j-2}y_j}-x_jy_{j+1}, \quad h=\underline{\prod_{\substack{i\in [0,n)\\i \text{ odd}}}y_i}-\prod_{\substack{i\in [0,n)\\i \text{ even}}}y_i.$$ Again, we need to consider only the cases when the leading terms of these binomials are not coprime. They cannot have a common $x$-factor, so Case~1 of Step~2 of \Cref{GB1} cannot happen here. If the two leading terms are not coprime, it has to be because of a common $y$-factor, and this brings us to Case~2 of Step~2 of \Cref{GB1}. Then, as in the proof of \Cref{GB1}, $j$ needs to be an odd number, in our case lying in $[0,n)$.  We once again see that $S(f_j,g_s)=x_{n-2}S(f_j,h)$. Following the proof of Case~2 of Step~2 of \Cref{GB1} (with $n=2s$ and $k=s$), we see that $x_{2k-2}$ (which now becomes $x_{n-2}$) is not involved in any reduction except in the step where $x_{2k-2}x_{2k-1}y_{2k}$ reduces $x_{n-2}x_{n-1}y_0$. In our case, however, this is actually no reduction since $n=2k$. Therefore, performing the same chain of reductions we will reduce $S(f_j,h)$ to $0$.
\end{proof}

\emph{\Cref{shan}.}
Let $<$ be the product order given in \Cref{GB1}.
Then the polynomials $$f_1, \ldots, f_{n-1},g_1, \ldots, g_{\frac{n}{2}-1}, h_1, \ldots, h_{\frac{n}{2}-1}$$ form a Gröbner basis for $L+HT$. Here $f_j=\underline{x_{\frac{n}{2}+j}y_j}-x_jy_{j+1}$,
$$g_k=\underline{y_k\prod_{i=0}^{k-1}x_i}-y_0\prod_{i=\frac{n}{2}}^{k-1+\frac{n}{2}}x_i,$$ 
and $h_l=\underline{y_ly_{l+\frac{n}{2}}}-y_0y_{\frac{n}{2}}$. As before, $L=(f_1,\ldots, f_{n-1},g_1,\ldots, g_{\frac{n}{2}-1})$ and $H=(h_1,\ldots,h_{\frac{n}{2}-1})$.

\begin{proof}[Proof of \Cref{shan}.]
The proof  will be divided into the following steps.

Step 1: \emph{The $S-$pairs $S(f_{j_1},f_{j_2})$ and $S(h_{l_1},h_{l_2})$ all reduce to $0$.}
This is obvious since the leading terms of $f_{j_1}$ and $f_{j_2}$ (resp. $h_{l_1}$ and $h_{l_2}$) are coprime for $j_1\not=j_2$ (resp. $l_1\not=l_2$).

Step 2: \emph{The $S-$pairs $S(g_{k_1},g_{k_2})$ all reduce to $0$.}
Let $1\leq k-s<k\leq\frac{n}{2}-1$. We would like to show that $S(g_k, g_{k-s})$ reduces to $0$. 
\begin{align*}
S(g_k,g_{k-s})&=S\left(\underline{y_k\prod_{i=0}^{k-1}x_i}-y_0\prod_{i=\frac{n}{2}}^{k-1+\frac{n}{2}}x_i,\quad\underline{y_{k-s}\prod_{i=0}^{k-s-1}x_i}-y_0\prod_{i=\frac{n}{2}}^{k-s-1+\frac{n}{2}}x_i\right)\\
&=g_k y_{k-s}-g_{k-s} y_k\prod_{i=k-s}^{k-1}x_i\\
&=y_0 y_k\prod_{i=k-s}^{k-1}x_i\prod_{i=\frac{n}{2}}^{k-s-1+\frac{n}{2}}x_i-y_0 y_{k-s}\prod_{i=\frac{n}{2}}^{k-1+\frac{n}{2}}x_i\\
&=y_0 \prod_{i=\frac{n}{2}}^{k-s-1+\frac{n}{2}}x_i\left(y_k\prod_{i=k-s}^{k-1}x_i-\underline{y_{k-s}\prod_{i=k-s+\frac{n}{2}}^{k-1+\frac{n}{2}}x_i}\right).
\end{align*}

Therefore, reducing the leading term of the binomial in parentheses using $f_{k-s},f_{k-s+1},\ldots f_{k-1}$, we obtain:
\begin{align*}
&y_{k-s}\prod_{i=k-s+\frac{n}{2}}^{k-1+\frac{n}{2}}x_i=\left(y_{k-s}x_{k-s+\frac{n}{2}}\right)\prod_{i=k-s+1+\frac{n}{2}}^{k-1+\frac{n}{2}}x_i\xrightarrow{f_{k-s}}x_{k-s}y_{k-s+1}\prod_{i=k-s+1+\frac{n}{2}}^{k-1+\frac{n}{2}}x_i\\&=x_{k-s}\left(y_{k-s+1}x_{k-s+\frac{n}{2}+1}\right)\prod_{i=k-s+2+\frac{n}{2}}^{k-1+\frac{n}{2}}x_i\xrightarrow{f_{k-s+1}}\ldots \xrightarrow{f_{k-1}}x_{k-s}\cdots x_{k-1}y_k=y_k\prod_{i=k-s}^{k-1}x_i.  
\end{align*}
Alternatively, one can do one reduction step (by $f_{k-s}$) and notice that the obtained monomial is $x_{k-s}$ times the monomial we are reducing, with $s\to s-1$. The rest follows easily by induction.
Therefore, $S(g_k, g_{k-s})$ reduces to $0$.

Step 3: \emph{The $S-$pairs $S(f_j, g_k)$ all reduce to $0$.}
Recall that $f_j=\underline{x_{\frac{n}{2}+j}y_j}-x_jy_{j+1}$ and
$g_k=\underline{y_k\prod_{i=0}^{k-1}x_i}-y_0\prod_{i=\frac{n}{2}}^{k-1+\frac{n}{2}}x_i$. 
We need to consider only the cases when the leading terms of these binomials are not coprime. This can happen in the following two cases:

Case 1: The two leading terms are not coprime because of a common $x-$factor, in other words, 
 $0\leq j+\frac{n}{2}\leq k-1$. Then we have:
    \begin{align*}
        S(f_j,g_k)&=S\left(\underline{x_{\frac{n}{2}+j}y_j}-x_jy_{j+1},\quad\underline{y_k\prod_{i=0}^{k-1}x_i}-y_0\prod_{i=\frac{n}{2}}^{k-1+\frac{n}{2}}x_i\right)\\
        &=f_k\cdot y_k\prod_{i=0}^{\frac{n}{2}+j-1}x_i\prod_{i=\frac{n}{2}+j+1}^{k-1}x_i-g_k\cdot y_j\\
        &=x_j\left(y_0y_j\prod_{i=\frac{n}{2}}^{j-1}x_i \prod_{i=j+1}^{k-1+\frac{n}{2}}x_i-\underline{y_ky_{j+1}\prod_{i=0}^{\frac{n}{2}+j-1}x_i\prod_{i=\frac{n}{2}+j+1}^{k-1}x_i}\right)
        \end{align*}
        Note that the underlined monomial is indeed the leading term of this binomial. In order to see this, it is enough to show that the underlined term is not divisible by $y_0$. Clearly, $k\not=0$ given the range where $g_k$'s are defined. If we assume $j+1=0$, then $0\le j+\frac{n}{2}\le k-1$ implies $0\le \frac{n}{2}-1\le k-1$, which is again a contradiction with the range where $g_k$'s are defined. We will first reduce a divisor of the leading term as follows:
        \begin{align*}
        y_ky_{j+1}\prod_{i=\frac{n}{2}+j+1}^{k-1}x_i&\xrightarrow{f_{j+1}}x_{j+1}y_ky_{j+2}\prod_{i=\frac{n}{2}+j+2}^{k-1}x_i\xrightarrow{f_{j+2}}x_{j+1}x_{j+2}y_ky_{j+3}\prod_{i=\frac{n}{2}+j+3}^{k-1}x_i\\
        &\xrightarrow{f_{j+3}}\cdots \xrightarrow{f_{k-1+\frac{n}{2}}}y_ky_{k+\frac{n}{2}}\prod_{i=j+1}^{k-1+\frac{n}{2}}x_i\xrightarrow{h_k}y_0y_{\frac{n}{2}}\prod_{i=j+1}^{k-1+\frac{n}{2}}x_i.
    \end{align*} 
Therefore, we have reduced the leading term in parentheses to $y_0y_{\frac{n}{2}}\prod_{i=j+1}^{k-1+\frac{n}{2}}x_i\prod_{i=0}^{\frac{n}{2}+j-1}x_i$. It is now sufficient to show that  $y_{\frac{n}{2}}\prod_{i=0}^{\frac{n}{2}+j-1}x_i$ can be reduced to $y_j\prod_{i=\frac{n}{2}}^{j-1}x_i$. But this is true since
\begin{align*}
y_{\frac{n}{2}}\prod_{i=0}^{\frac{n}{2}+j-1}x_i&\xrightarrow{f_{\frac{n}{2}}}x_{\frac{n}{2}}y_{\frac{n}{2}+1}\prod_{i=1}^{\frac{n}{2}+j-1}x_i\xrightarrow{f_{\frac{n}{2}+1}}x_{\frac{n}{2}}x_{\frac{n}{2}+1}y_{\frac{n}{2}+2}\prod_{i=2}^{\frac{n}{2}+j-1}x_i\\&\xrightarrow{f_{\frac{n}{2}+2}}\cdots\xrightarrow{f_{j-1}}y_j\prod_{i=\frac{n}{2}}^{j-1}x_i.
\end{align*}
   Therefore, $S(f_j,g_k)$ reduce to $0$ in this case.

Case 2: The two leading terms are not coprime because of a common $y-$factor, in other words, $j=k\in \{1,\ldots, \frac{n}{2}-1\}$. In this case we have:
\begin{align*}
S(f_k,g_k)&=S\left(\underline{x_{\frac{n}{2}+k}y_k}-x_ky_{k+1}, \quad\underline{y_k\prod_{i=0}^{k-1}x_i}-y_0\prod_{i=\frac{n}{2}}^{k-1+\frac{n}{2}}x_i\right)\\
&=f_k\cdot \prod_{i=0}^{k-1} x_i-g_k\cdot x_{\frac{n}{2}+k}=y_0\prod_{i=\frac{n}{2}}^{k+\frac{n}{2}}x_i-\underline{y_{k+1}\prod_{i=0}^k x_i}.
\end{align*} 
Note that this binomial equals $-g_{k+1}$, unless $k=\frac{n}{2}-1$. If $k=\frac{n}{2}-1$, we have:

\begin{align*}
y_{\frac{n}{2}}\prod_{i=0}^{\frac{n}{2}-1} x_i&=x_0y_{\frac{n}{2}}\prod_{i=1}^{\frac{n}{2}-1}x_i\xrightarrow{f_{\frac{n}{2}}}x_{\frac{n}{2}}y_{\frac{n}{2}+1}\prod_{i=1}^{\frac{n}{2}-1}x_i=x_{\frac{n}{2}}(x_1y_{\frac{n}{2}+1})\prod_{i=2}^{\frac{n}{2}-1}x_i\\&\xrightarrow{f_{\frac{n}{2}+1}}x_{\frac{n}{2}}x_{\frac{n}{2}+1}y_{\frac{n}{2}+2}\prod_{i=2}^{\frac{n}{2}-1}x_i=x_{\frac{n}{2}}x_{\frac{n}{2}+1}(x_2y_{\frac{n}{2}+2})\prod_{i=3}^{\frac{n}{2}-1}x_i\\&\xrightarrow{f_{\frac{n}{2}+2}}
\ldots \xrightarrow{f_{n-1}}x_{\frac{n}{2}}x_{\frac{n}{2}+1}\cdots x_{n-1}y_0=y_0\prod_{i=\frac{n}{2}}^{n-1}x_i,
\end{align*}

which is exactly what we need.
Therefore, $S(f_j,g_k)$ all reduce to $0$.

Step 4: \emph{The $S-$pairs $S(g_k,h_l)$ all reduce to $0$.}
Notice that the leading terms of $g_k$ and $h_l$ are coprime except $k=l$. In this case we have:

\begin{align*}
S(g_l,h_l)&=S\left(\underline{y_l\prod_{i=0}^{l-1}x_i}-y_0\prod_{i=\frac{n}{2}}^{l-1+\frac{n}{2}}x_i,\quad\underline{y_l y_{l+\frac{n}{2}}}-y_0y_{\frac{n}{2}}\right)\\
&=g_l y_{l+\frac{n}{2}}-h_l \prod_{i=0}^{l-1}x_i
=y_0\left(\underline{y_{\frac{n}{2}}\prod_{i=0}^{l-1}x_i}-y_{l+\frac{n}{2}} \prod_{i=\frac{n}{2}}^{l-1+\frac{n}{2}}x_i\right).
\end{align*} 

Note that the leading term of the binomial in parentheses is very similar to the leading term of the $S$-pair in the previous Step~3 (the case $k=\frac{n}{2}-1$). We do the exact same reduction of this monomial, but with fewer steps. We use $f_{\frac{n}{2}},f_{\frac{n}{2}+1},\ldots,f_{\frac{n}{2}+l-1}$ and from Step~3 we see that after these reductions we will obtain $x_{\frac{n}{2}}\cdots x_{\frac{n}{2}+l-1}y_{\frac{n}{2}+l}=y_{l+\frac{n}{2}}\prod_{i=\frac{n}{2}}^{l-1+\frac{n}{2}}x_i$, which is exactly what we need. We conclude that all $S(g_k,h_l)$ reduce to $0$.

Step 5: \emph{The $S-$pairs $S(f_j,h_l)$ all reduce to $0$.}
Notice that the leading terms of $f_j$ and $h_l$ are coprime except the case $j=l\in\{1,\ldots, \frac{n}{2}-1\}$ and the case $l\in\{1,\ldots, \frac{n}{2}-1\}$, $j=l+\frac{n}{2}$. In the former case we have:
\begin{align*}
S(f_l,h_l)&=S\left(\underline{x_{\frac{n}{2}+l}y_l}-x_ly_{l+1}, \quad\underline{y_ly_{l+\frac{n}{2}}}-y_0y_{\frac{n}{2}}\right)=f_ly_{l+\frac{n}{2}}-h_lx_{\frac{n}{2}+l}\\&=x_{\frac{n}{2}+l}y_0y_{\frac{n}{2}}-\underline{x_ly_{l+1}y_{l+\frac{n}{2}}}
\xrightarrow{f_{l+\frac{n}{2}}}x_{\frac{n}{2}+l}y_0y_{\frac{n}{2}}-x_{l+\frac{n}{2}}y_{l+1+\frac{n}{2}}y_{l+1}
=-x_{\frac{n}{2}+l}h_{l+1}.
\end{align*}

Note that the last equality makes sense unless $l=\frac{n}{2}-1$, in which case $h_{l+1}$ is not defined. However, in this case the left hand side of this equality equals $0$.
Similarly, in the latter case $S(f_{l+\frac{n}{2}},h_l)$ reduces to 
$x_lh_{l+1}$ using $f_l$ (and the same remark about $l=\frac{n}{2}-1$ applies). Therefore, $S(f_j,h_l)$ all reduce to $0$.

\end{proof}

\vspace{3mm}
\textsc{Oleksandra Gasanova, Fakultät für Mathematik, Universität Duisburg-Essen, 45117 Essen, Germany}

\textit{Email address:} \texttt{oleksandra.gasanova@uni-due.de}

\textsc{Jürgen Herzog, Fakultät für Mathematik, Universität Duisburg-Essen, 45117 Essen, Germany}

\textit{Email address:} \texttt{juergen.herzog@uni-essen.de}

\textsc{Jiawen Shan, School of Mathematical Sciences, Soochow University, 215006 Suzhou, P. R. China}

\textit{Email address:} \texttt{ShanJiawen826@outlook.com}


\begin{thebibliography}{9999}
\bibitem{Bd01} P. Brumatti, A. F. da Silva, \textit{On the symmetric and Rees algebras of (n, k)-cyclic ideals},  16th School of Algebra, Part II (Portuguese) (Brasilia, 2000). Mat. Contemp. \textbf{21} (2001), 27--42.\par
\bibitem{BH} W. Bruns, J. Herzog, \textit{Cohen-Macaulay rings}, Cambridge University Press, 1998. \par
\bibitem{CD99} A. Conca, E. De Negri, \textit{M-sequences, graph ideals and ladder ideals of linear type}, J. Algebra \textbf{211} (1999), 599--624.\par
\bibitem{Eis95} D. Eisenbud, \textit{Commutative Algebra: with a View Toward Algebraic Geometry},  Springer-Verlag, Inc., New York, 1995.\par
\bibitem{HH} J. Herzog, T. Hibi, \textit{Monomial Ideals}, Graduate Texts in Mathematics, vol. 260. Springer-Verlag, 2011. \par
\bibitem{HHV05} J. Herzog, T. Hibi and M. Vladoiu, \textit{Ideals of fiber type and polymatroids}, Osaka J. Math. \textbf{42} (2005), 807--829.\par
\bibitem{Ho} M. Hochster, \textit{Rings of invariants of tori, Cohen-Macaulay rings generated by monomials, and polytopes}, Ann. of Math. \textbf{96} (1972), 318-337.
\bibitem{circulant} A. W. Ingleton, \textit{The rank of circulant matrices}, Journal of the London Mathematical Society, \textbf{1.4} (1956), 445-460.
\bibitem{Stanley} R.P. Stanley, \textit{Hilbert functions of graded algebras}, Adv. Math. \textbf{28} (1978) 57–83.
\bibitem{St}  B. Sturmfels, \textit{Gröbner Bases and Convex Polytopes}, Amer. Math. Soc., Providence, RI, 1995.
\bibitem{Vil90} R. Villarreal, \textit{Cohen-Macaulay graphs}, Manuscripta. Math.  \textbf{66} (1990), 277--293.\par
\bibitem{Vil01} R. Villarreal, \textit{Monomial Algebras}, Marcel Dekker, Inc., New York, 2001.\par
\end{thebibliography}
\end{document}